%% file: main.tex
\documentclass[12pt,oneside]{article}
\input{template.tex}
\title{The algebraic K-theory of the K(1)-local sphere via TC}
\author{Ishan Levy\thanks{The author is supported by the NSF Graduate Research Fellowship under Grant No. 1745302.}}
\usepackage{setspace}
\begin{document}
	\date{}
	\maketitle
	\begin{abstract}
		We describe the algebraic $K$-theory of the $K(1)$-local sphere and the category of type 2 finite spectra in terms of $K$-theory of discrete rings and topological cyclic homology. We find an infinite family of $2$-torsion classes in the $K_0$ of type 2 spectra at the prime $2$, and explain how to construct representatives of these $K_0$ classes.
	\end{abstract}
\begin{center}
	\includegraphics[scale = .35]{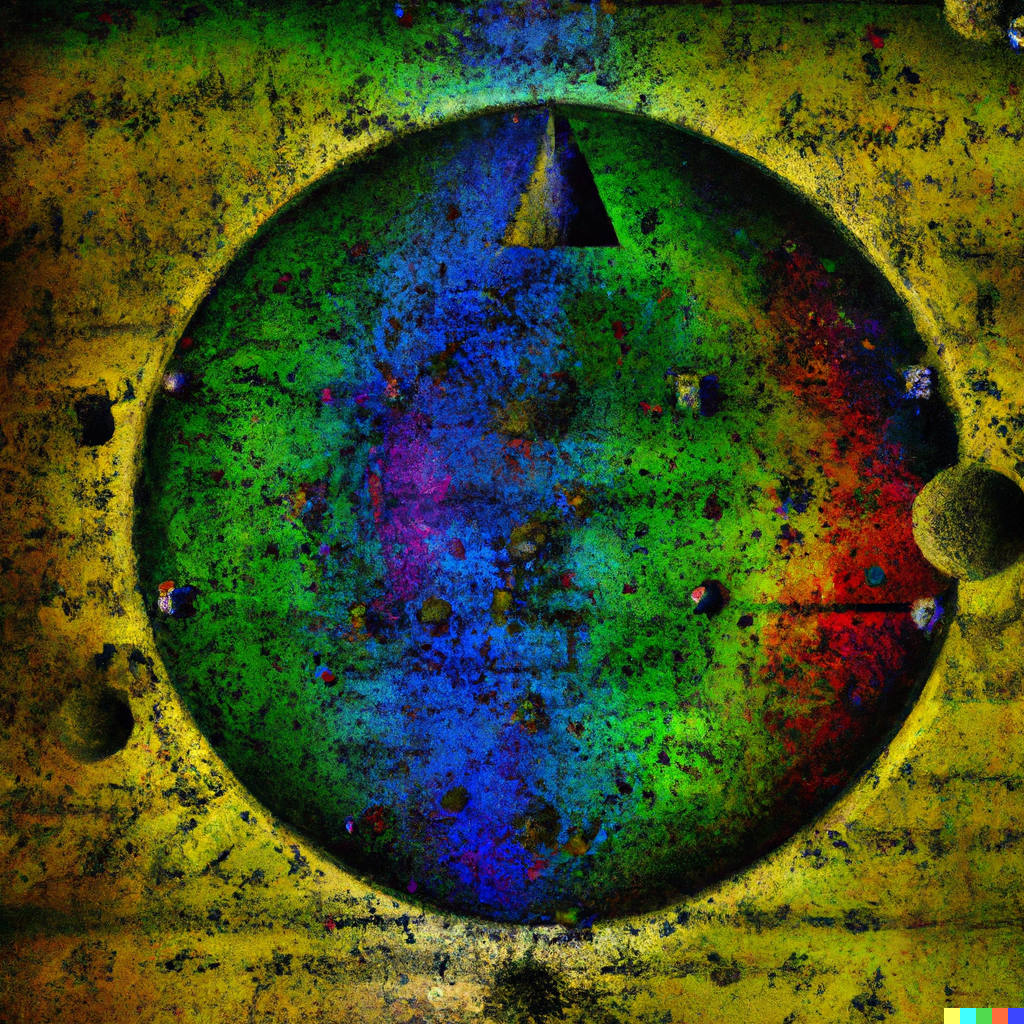}\footnote{image generated by OpenAI DALL·E 2}
\end{center}
		\tableofcontents
	\section{Introduction}
	
	Fix a prime $p$ and let $\Sp_p^{\diamondsuit}$ be the category of dualizable $p$-complete spectra. One of the fundamental results in stable homotopy theory is the thick subcategory theorem of Hopkins and Smith \cite{hopkins1998nilpotence}, which says that every nonzero thick subcategory of $\Sp_p^{\diamondsuit}$ is one of the categories $\Sp^{\omega}_{\geq n}$\footnote{when $n=0$, this denotes the category $\Sp_p^{\diamondsuit}$.} of finite spectra of type at least $n$ for some $n\geq 0$. A modern interpretation of this result is the statement that the Balmer spectrum of $\Sp_p^{\diamondsuit}$ agrees with the Zariski spectrum of $\cM_{\text{fg},p}$, the moduli stack of $p$-typical formal groups.
	
	In this paper, we study a subtle additional structure on the Balmer spectrum of $\Sp_p^{\diamondsuit}$, namely its sheaf of algebraic $K$-theory.\footnote{It is not important that we work in a $p$-completed setting, it is just convenient, as the chromatic localizations $L_n^f$ don't affect the rationalization or $\ell$-adic completions.} To the open set corresponding to the height $\leq n$ locus, this sheaf takes the value $K(L_n^f\SP_p)$, and on global sections, it  is $K(\SP_p)$. This sheaf was first considered in \cite{waldhausen1984algebraic}, where fundamental localization sequences relating $K(L_n^f\SP_p)$ to the $K$-theory of the monochromatic layers $K(\Sp_{T(n)}^{\omega})$ were observed. Thomason in \cite{thomason1997classification} showed that understanding the homotopy groups of the sheaf in low degrees would allow one to refine the thick subcategory theorem, and classify \textit{stable} subcategories\footnote{In contrast to thick subcategories, stable subcategories may not be closed under retracts} of $\Sp_{p}^{\diamondsuit}$.
	
	The only case in which $K(L_n^f\SP_p)$ is well understood is the case $n=0$, where it is $K(\QQ_p)$.\footnote{see \cite{weibel2005algebraic} for a discussion of what is known.} For $n\geq 1$, essentially the only thing previously known about $K(L_n^f\SP_p)$ was its chromatic height, because of redshift \cite{ausoni2000algebraic,clausen2020descent,land2020purity,jeremydylanredshift,allenredshift}. Further information was previously out of reach: for example, no $K$ group was previously known.
	
	In contrast, $K(\SP_p)$ is now well understood. The reason is that $\SP_p$ is a connective ring, and the Dundas--Goodwillie--McCarthy (or DGM) theorem  \cite{dundas2012local,raskin2018dundas} gives a pullback square for any connective ring $R$ of the form 
	\begin{center}
		\begin{tikzcd}
			\pullback K(R)\ar[r]\ar[d] &\TC(R) \ar[d]\\
			K(\pi_0R)\ar[r] & \TC(\pi_0R)
		\end{tikzcd}
	\end{center}
	where the horizontal maps are the cyclotomic trace. This largely reduces the computation of $K$-theory to understanding $\TC$ and the $K$-theory of discrete rings, both of which are usually more tractable invariants. Finally one must analyse the cyclotomic trace and reconstruct the $K$-theory of a connective ring from its constituent pieces in the pullback square. For the sphere, this is carried out in \cite{rognes2003smooth,blumberg2019homotopy}\footnote{to understand the homotopy groups of the $K$-theory of the sphere from the pullback square, a finite generation result of Dwyer as well as an analysis of the arithmetic fracture square are used in \cite{blumberg2019homotopy}.}. 
	
	The rings $L_n^f\SP_p$ are not connective, so DGM cannot directly be applied to compute their $K$-theory. Here we show nevertheless that $K(L_1^f\SP_p)$ can be described in terms of $\TC$ and $K$-theory of discrete rings, answering Problem 2.6 of \cite{antieauopenproblemsring}. To state our result, we need to introduce the ring $j_{\zeta}$ below.
	
	\begin{definition}
		Let $j_{\zeta}$ be the $\EE_{\infty}$-ring $\ell_p^{h\ZZ}$ for $p>2$, and $ko_2^{h\ZZ}$ if $p=2$. Here $\ell_p$ is the $p$-completed Adams summand of connective topological $K$-theory and $ko_2$ is $2$-completed connective real topological $K$-theory, and the $\ZZ$ action comes from the Adams operation $\Psi^{1+p}$.
	\end{definition} 
		The underlying spectrum of $j_{\zeta}$ can be described as the $-1$-connective cover of the $K(1)$-local sphere.
	\begin{customthm}{A}\label{thm:K1localsphere}
		$K(L_1^f \SP_p) \cong K(L_{K(1)}\SP)$, there is a cofibre sequence split on $\pi_*$
		\begin{center}
			\begin{tikzcd}
			K(j_{\zeta}) \ar[r] & K(L_{K(1)}\SP) \ar[r] & \Sigma K(\FF_p)
			\end{tikzcd}
		\end{center}
		and a pullback square
\begin{center}
	\begin{tikzcd}
		\pullback K(j_{\zeta}) \ar[r] \ar[d] & \TC(j_{\zeta}) \ar[d]\\
		K(\ZZ_p) \ar[r]  & \TC(\ZZ_p^{h\ZZ})
	\end{tikzcd}
\end{center}
	Let $F$ be the fibre of the map $\TC(j_{\zeta}) \to \TC(\ZZ_p^{h\ZZ})$. Then $F[\frac 1 p] = 0$. For $p>2$, $F$ is $(2p-2)$-connective and $\pi_{2p-2}(F/p) \cong \bigoplus_0^{\infty}\FF_p$. For $p=2$, $F$ is $1$-connective and $\pi_1F\cong \bigoplus_0^{\infty}\FF_2$.
	\end{customthm}

In particular, even for the ring $L_1^f\SP$, whose localizations at the primes other than $p$ agree with that of the sphere, its $K$-theory is not degree-wise finitely generated! This is in sharp contrast to $K(\SP)$ and $K(L_0^f\SP) = K(\SP[\frac 1 p])$, which are degree-wise finitely generated.

The proof of the cofibre sequence in \Cref{thm:K1localsphere}, which is carried out in \cref{sec:devissage}, comes from analysing the localization sequence 

\begin{center}
	\begin{tikzcd}
			\Mod(j_{\zeta})^{\omega}\otimes \Sp_{\geq 2}^{\omega}\ar[r] & \Mod(j_{\zeta})^{\omega} \ar[r] &\Mod(L^f_1j_{\zeta})^{\omega}
	\end{tikzcd}
\end{center}
and showing that on $K$-theory, it induces the desired cofibre sequence. Because $L^f_1j_{\zeta} = L_{K(1)}\SP$, the only substantial claim is that $K(\FF_p) \cong K(\Mod(j_{\zeta})^{\omega}\otimes \Sp_{\geq 2}^{\omega})$. To obtain this we choose a particularly good generator of the category, namely $j_{\zeta} \otimes Z$, where for $p>2$, $Z$ is the Smith--Toda complex $\SP/(p,v_1)$ constructed by Toda \cite{toda1971spectra}, and for $p=2$ it is $\SP/(2,\eta,v_1)$, the type $2$ spectrum constructed by Davis and Mahowald \cite{davis1981v}. We then show that the endomorphism ring of $j_{\zeta} \otimes Z$ is coconnective with $\pi_0 = \FF_p$, so that we can conclude by applying the devissage result of \cite{kcoconn} that $K(\Mod(j_{\zeta})^{\omega}\otimes \Sp_{\geq 2}^{\omega}) \simeq K(\FF_p)$.

The other main claim in \Cref{thm:K1localsphere} is the pullback square, which allows us to understand $K(j_{\zeta})$. This is almost an immediate consequence of \Cref{thm:dgmfixedpoints} below, which extends the DGM theorem to include the map $\ell_p^{h\ZZ} \to \ZZ_p^{h\ZZ}$. Recall that a \textit{truncating invariant} $E$ is a localizing invariant for which the map $E(R) \to E(\pi_0R)$ is an equivalence for any connective ring $R$. In this language, DGM says that the fibre of the cyclotomic trace is truncating.

\begin{customthm}{B}\label{thm:dgmfixedpoints}
	Let $f:R\to S$ be a map of connective $\EE_1$-rings with a $\ZZ$-action such that $f$ is $1$-connective. Then for any truncating invariant $E$, $E(R^{h\ZZ}) \to E(S^{h\ZZ})$ is an equivalence. Moreover, if $f$ is $n$-connective, then $\TC(R^{h\ZZ}) \to \TC(S^{h\ZZ})$ is too.
\end{customthm}

We also obtain the following variant:
\begin{customthm}{C}\label{thm:dgm-1conn}
	Let $R \to S$ be a $1$-connective map of $-1$-connective rings such that $\pi_{-1}R$ is a finitely generated $\pi_0R$-module. Then for any truncating invariant $E$, $E(R) \to E(S)$ is an equivalence.
\end{customthm}

%

The proof of \Cref{thm:dgmfixedpoints}, which can be found in \cref{sec:tc} is an application of the work of Land--Tamme on the $K$-theory of pullbacks. Namely, one has a pullback diagram

\begin{center}
	\begin{tikzcd}
		R^{h\ZZ} \ar[r]\ar[d] &R \ar[d]\\
	R	\ar[r] & R\times R
	\end{tikzcd}
\end{center}
Applying the main result of \cite{Land_2019}, one obtains a pullback square after applying any localizing invariant, where $R\times R$ is replaced by the ring $R\odot_{R^{h\ZZ}}^{R\times R}R$. The latter ring is connective, and comparing with the analogous construction for $S$ and using the pullback square and the fact that the invariant is truncating, one obtains the result.

In addition to \Cref{thm:K1localsphere}, we also obtain a similar formula for $K(\Sp_{T(1)}^{\omega})$ in \Cref{thm:compactk1local}, which we use in \cref{sec:eulerchar} to answer \cite[Problem 16.4]{hovey1999morava} at height 1. For $K(\Sp_{\geq2})$ we obtain the result below.

\begin{customthm}{D}\label{thm:type2}
	There is a fibre sequence $X \to K(\Sp_{\geq 2}) \to K(\FF_p)$ split on $\pi_*$, where $X$ is the total fibre of the square
	\begin{center}
		\begin{tikzcd}
			\TC(\SP_p)\ar[r]\ar[d] &\TC(\ZZ_p) \ar[d]\\
			\ar[r] \TC(j_{\zeta})& \TC(\ZZ_p^{h\ZZ})
		\end{tikzcd}
	\end{center}
\begin{itemize}
	\item For $p>2$, $X$ is $(2p-3)$-connective, so $K_0(\Sp_{\geq 2}) = \ZZ$ with generator $[\SP/(p,v_1)]$.
	\item For $p=2$, $X$ is connective with $\pi_0X = \bigoplus_0^{\infty}\ZZ/2$, and the torsion free quotient of $K_0(\Sp_{\geq 2})$ is generated by $[\SP/(2,\eta,v_1)]$.
\end{itemize}
	
\end{customthm}

In particular we find, contrary to our initial expectations that at the prime $2$ there are infinitely many $2$-torsion classes in $K_0(\Sp^{\omega}_{\geq2})$! As a corollary, we obtain a classification of dense stable subcategories of type $2$ spectra. A full stable subcategory $C' \subset C$ is \textit{dense} if the inclusion is an equivalence on idempotent completions.

\begin{corollary}(\Cref{cor:densesubcattype2})
	The dense stable subcategories of $\Sp_{\geq2}^{\omega}$ for $p>2$ are in bijection with subgroups of $\ZZ$, and the dense stable subcategories of $\Sp_{\geq2}^{\omega}$ at the prime $2$ are in bijection with subgroups of $\ZZ\oplus \bigoplus_0^{\infty}\FF_2$.
\end{corollary}

In \cref{sec:const}, we explain how to construct explicit spectra representing all of the $2$-torsion classes, but we briefly explain how to construct the first one here. We first choose a self map $v_1^4:\Sigma^8 \SP/2 \to \SP/2$. Because $\eta\sigma$ is $2$-torsion in $\pi_8\SP_{2}$, we can produce an extension of it to a map $\Sigma^8 \SP/2 \to \SP$. Let $\overline{\eta\sigma}$ be the composite 
$$\overline{\eta\sigma}:\Sigma^8 \SP/2 \rightarrow \SP \rightarrow \SP/2$$
Then $[\SP/(2,v_1^4+\overline{\eta\sigma})]-[\SP/(2,v_1^4)]$ represents the first $2$-torsion class in $K_0(\Sp_{\geq2})$.

%

%
%


We ask open questions throughout the paper related to this work. A particularly important one is the following:
\begin{question}
	What can be said about $\TC(j_{\zeta})$? For example, can its homotopy groups be computed, at least mod $(p,v_1)$ or $(p,v_1,v_2)$? 
\end{question}

As a first step to the above question, in forthcoming work joint with David Lee, we compute $\THH(j_{\zeta})$ mod $(p,v_1)$ at odd primes.

\begin{theorem}[Lee--Levy \cite{THHj}]
	For $p>2$, there is an isomorphism of graded rings 
	
	$$\pi_*\THH(j_{\zeta})/(p,v_1) \cong \pi_*\THH(\ell)/(p,v_1)\otimes \pi_*\mathrm{HH}(C^*(S^1;\FF_p)/\FF_p)$$
\end{theorem}

We don't know the extent to which the methods of this paper are capable of understanding higher height phenomena. Despite this, in forthcoming work joint with Robert Burklund, we completely compute the $K$-theory sheaf after inverting the prime $p$.

\begin{theorem}[Burklund--Levy \cite{pinverted}]\label{thm:pinverted}
	For $n\geq 1$, there are isomorphisms $$K(L_n\SP_p)[\frac 1 p]\cong K(L_n^f\SP_p)[\frac 1 p]\cong K(\ZZ_p)[\frac 1 p] \oplus \Sigma K(\FF_p)[\frac 1 p]$$ Moreover $K(\Sp_{\geq n})[\frac 1 p] \cong K(\FF_p)[\frac 1 p]$, and a generator of $K_0(\Sp_{\geq n})[\frac 1 p]$ is given by the class of a generalized Moore spectrum $[\SP/(p,v_1^{p^{i_1}},\dots, v_{n-1}^{p^{i_{n-1}}})]$.
\end{theorem}

%
%



\subsection*{Conventions}
We assume the reader is familiar with higher algebra and algebraic $K$-theory. Some conventions we use are:

\begin{itemize}
\item The term category will refer to an $\infty$-category as developed by Joyal and Lurie.
\item $\Map_C(a,b)$ denotes the space of maps from $a$ to $b$ in a category $C$. $C$ is omitted from the notation when it is clear from context.
\item Similarly, in a stable category $C$, $\map_C(a,b)$ denotes the mapping spectrum.
\item For an $\EE_1$-algebra $R$, $\Mod(R)$ refers to its category of left modules.
\item For an $\EE_1$-algebra $R$ and a stable category $C$, $R\otimes C$ is shorthand for either $\Mod(R)\otimes C$ if $C$ is presentable or $\Mod(R)^{\omega}\otimes C$ if $C$ is small.\footnote{there is no ambiguity because if $C$ is stable, presentable, and small, it must be zero, in which case $\Mod(R)\otimes C$ and $\Mod(R)^{\omega}\otimes C$ are both the zero category.}
\item We use $K(-)$ for \textit{nonconnective} $K$-theory. For a compactly generated stable category $C$, $K(C)$ will mean $K(C^{\omega})$.
\item $\Uloc$ and $\Uadd$ denote the versions of the universal localizing and additive invariants of \cite{BGT} that do not preserve any kind of filtered colimits.
\item We use $x_n$ for a polynomial generator in degree $n$ and $\epsilon_n$ for an exterior generator in degree $n$. As an example, $\SP[x_n]$ is the free $\EE_1$-algebra on a class in degree $n$.
\end{itemize}

\subsection*{Acknowledgements}
I would like to thank Andrew Blumberg, Robert Burklund, Sanath Devalapurkar, David Gepner, Jeremy Hahn, Markus Land, David Jongwon Lee, Akhil Mathew, Lennart Meier, and Georg Tamme for helpful conversations related to this work. I am especially grateful to Mike Hopkins for getting me started with this project, and for his encouragement and advice. I would like to thank Andrew Blumberg, Jeremy Hahn, and especially Robert Burklund for their helpful comments on earlier drafts.
\section{Localization sequences and devissage}\label{sec:devissage}

The main goal of this section is to prove the parts of \Cref{thm:K1localsphere} that come from localization sequences and devissage, namely \Cref{prop:L1LK1}, \Cref{prop:K1localsequence}, and \Cref{lem:nullona1invkthy} below. These results allow us to reduce the study of objects such as $K(L_1^f\SP)$, $K(\Sp_{\geq2}^{\omega})$, and $K(\Sp_{K(1)}^{\omega})$ to the study of $K(j_{\zeta})$. The key tool here is devissage in the form given in \cite{kcoconn}. 

\begin{theorem}[\cite{kcoconn}]\label{thm:devissage}
	If $R$ is a coconnective ring with $\pi_0$ regular, and $\pi_{-i}$ has tor dimension $< i$ over $\pi_0$, then the connective cover map $\pi_0R \to R$ is an equivalence on $K$-theory.
\end{theorem}

To begin, recall that there is a localization sequence

\begin{equation}\label{eqn:locseq}
	\Sp_{\geq n+1} \to \Sp \to L_n^f\Sp
\end{equation}

Our propositions are obtained by tensoring this with the rings in question.

\begin{proposition}\label{prop:L1LK1}
	The natural map $K(L_1^f\SP_p) \to K(L_{K(1)}\SP)$ is an equivalence.
\end{proposition}

\begin{proof}
	Tensoring the localization sequence (\ref{eqn:locseq}) for $n=0$ with the map $L_1^f\SP_p \to L_{K(1)}\SP$, we get a map of localization sequences
	
	\begin{center}
		\begin{tikzcd}
			\Mod(L_1^f\SP_p)^{p-\text{nil}}\ar[r]\ar[d] &\Mod(L_1^f\SP_p) \ar[r]\ar[d] & \Mod(L_0^fL_1^f\SP_p)\ar[d] \\
			\Mod(L_{K(1)}\SP)^{p-\text{nil}}\ar[r] &	\Mod(L_{K(1)}\SP) \ar[r] &	\Mod(L_0^fL_{K(1)}\SP)
		\end{tikzcd}
	\end{center}
	The category in the top left is the category of $T(1)$-local spectra, whereas the one in the bottom left is the category of $K(1)$-local spectra. By the telescope conjecture at height $1$ \cite{miller1981relations,mahowald1981bo}, these two categories agree, so the left vertical map is an equivalence.
	
	$L_0^f$ is just inverting $p$, and $L_0^fL_1^f\SP_p = L_0^f\SP_p = \QQ_p$. $L_0^fL_{K(1)}\SP$ is the ring $\QQ_p[\epsilon_{-1}]$, ($\epsilon_{-1}$ is usually called $\zeta$). It follows that the right vertical map is a connective cover map, and so applying \Cref{thm:devissage}, it is an equivalence on $K$-theory.
	
	Since $K$-theory is a localizing invariant, the middle vertical map is also an equivalence on $K$-theory.
\end{proof}

The next proposition is somewhat more subtle, because the ring $j_{\zeta}$ is not regular in the sense of \cite{ncgstuff,kcoconn}. Nevertheless, a formal neighborhood of its height $\geq2$ locus is regular, which is all that is needed.

\begin{proposition}\label{prop:K1localsequence}
	There is a cofibre sequence
	$K(\FF_p) \to K(j_{\zeta}) \to K(L_{K(1)}\SP)$.
\end{proposition}

\begin{proof}
	Tensoring the localization sequence \ref{eqn:locseq} for $n=1$ (relative to $\Sp^{\omega}$) with $\Mod(j_{\zeta})^{\omega}$, we get a cofibre sequence
	
	\begin{center}
		\begin{tikzcd}
			\Mod(j_{\zeta})^{\omega}\otimes \Sp_{\geq 2}^{\omega}\ar[r] & \Mod(j_{\zeta})^{\omega} \ar[r] &\Mod(L^f_1j_{\zeta})^{\omega}
		\end{tikzcd}
	\end{center}
We claim that $L_1^fj_{\zeta} = L_{K(1)}\SP$. Indeed, it is clear that $L_{K(1)}j_{\zeta}= L_{K(1)}\SP$, and $j_{\zeta}$ and $L_{K(1)}\SP$ are rationally both $\QQ_p[\epsilon_{-1}]$. 
%
	
	It remains to identify $K(\Mod(j_{\zeta})^{\omega}\otimes \Sp_{\geq 2}^{\omega})$ with $K(\FF_p)$. By the thick subcategory theorem, $\Sp_{\geq 2}$ is generated by any type $2$ spectrum $Z$, so $\Mod(j_{\zeta})^{\omega}\otimes \Sp_{\geq 2}^{\omega}$ is generated by $j_{\zeta}\otimes Z$ for such $Z$. 
	
	Let $Z$ denote the type $2$ spectrum which for $p>2$ is the Smith-Toda complex $\SP/(p,v_1)$ constructed in \cite{toda1971spectra} and for $p=2$, is the type 2 complex $\SP/(2,\eta,v_1)$\footnote{There are actually 8 distinct $v_1$ self maps on $\SP/(2,\eta)$ and 4 nonisomorphic $\SP/(2,\eta,v_1)$s, but this is irrelevant here: for example all choices of $\SP/(2,\eta,v_1)$ become isomorphic after basechange to $j_{\zeta}$. This follows from the proof of this proposition, since they are in the heart of a bounded $t$-structure after basechange and have only one cell in degree $0$, so there is no obstruction to producing an isomorphism between the different versions.} constructed in \cite{davis1981v}. The key property of $Z$ is that $Z\otimes \ko = \FF_2$ at the prime $2$ and $Z\otimes \ell = \FF_p$ for $p>2$. It follows that $Z\otimes j_{\zeta} = \FF_p^{h\ZZ}$, which is in particular coconnective with $\pi_0 = \FF_p$.
	
	Since $Z$ has only one cell in dimension $0$ and the rest in positive degrees, $\End(Z\otimes j_{\zeta}) \cong Z\otimes Z^{\vee} \otimes j_{\zeta} $ is also coconnective with $\pi_0 = \FF_p$. We learn from Morita theory and \Cref{thm:devissage} that $K(j_{\zeta}\otimes \Sp_{\geq 2}) \cong K(\End_{j_{\zeta}}(j_{\zeta}\otimes X)) \cong K(\FF_p)$.
	
\end{proof}

	The following lemma is necessary to get that the cofibre sequences in \Cref{thm:K1localsphere} and \Cref{thm:compactk1local} are split on $\pi_*$.
\begin{lemma}\label{lem:nullona1invkthy}
	$K(\ZZ_p^{h\ZZ}\otimes \Sp_{\geq1}^{\omega}) \cong K(\FF_p)$ and the composite $$\Mod(\FF_p)^{\omega} \to j_{\zeta}\otimes \Sp_{\geq 2}^{\omega} \to j_{\zeta}\otimes \Sp_{\geq1}^{\omega} \to \ZZ_p^{h\ZZ}\otimes \Sp_{\geq1}^{\omega}$$ is null after applying $\Uadd$.
	
\end{lemma}

\begin{proof}
	$\cof p \in \Mod(\ZZ_p^{h\ZZ})^{p-\text{nil}}$ has a coconnective endomorphism ring with $\pi_0 = \FF_p$, so $K(\Mod(\ZZ_p^{h\ZZ})^{p-\text{nil}}) = K(\FF_p)$ by \Cref{thm:devissage}. It remains to prove the second claim.
	
	Let $R$ be $\ell_p$ or $\ko_2$ depending on the prime so that $R^{h\ZZ} = j_{\zeta}$. Then there are $\ZZ$-equivariant $\EE_{\infty}$-maps $\pi:R\to \FF_p$ and $f:R\to \ZZ_p$, and $\FF_p$ is perfect over $R$ as it is $R\otimes Z$, where $Z$ is as in the proof of \Cref{prop:K1localsequence}. We use the same names to denote the induced maps on $\ZZ$-homotopy fixed points. There is also the connective cover map $g: \FF_p \to \FF_p^{h\ZZ}$.
	
	Taking $\ZZ$-homotopy fixed points, we obtain a diagram
	
	\begin{center}
		\begin{tikzcd}
			\Mod(\ZZ_p^{h\ZZ})^{\omega,p-\text{nil}}& \ar[l,"f^*"]\Mod(j_{\zeta})^{\omega}\otimes \Sp^{\omega}_{\geq2}\ar[r,"\pi^*", shift left = 1.5] & \ar[l,"{\pi_*}", shift left = 1.5]\Mod(\FF_p^{h\ZZ} )^{\omega}
		\end{tikzcd}
	\end{center}

	Here $\pi_*$ is the right adjoint of $\pi^*$, which exists since $\FF_p^{h\ZZ}$ is perfect over $j_{\zeta}$. We also have the map $g^*:\Mod(\FF_p)^{\omega} \to \Mod(\FF_p^{h\ZZ})^{\omega}$, and we can rephrase our lemma as saying that $f^*\circ \pi_* \circ g^*$ is null on $\Uadd$.
	
	We will in fact just show that $f^*\circ \pi_*$ is null on $\Uadd$. To do this, that composite is given by tensoring with the $\ZZ_p^{h\ZZ}-\FF_p^{h\ZZ}$-bimodule $\ZZ_p^{h\ZZ}\otimes_{j_{\zeta}} \FF_p^{h\ZZ}$. We have a chain of equivalences $$\ZZ_p^{h\ZZ}\otimes_{j_{\zeta}} \FF_p^{h\ZZ} \cong \ZZ_p^{h\ZZ}\otimes Z \cong (\ZZ_p\otimes Z)^{h\ZZ} \cong (\ZZ_p\otimes_R\FF_p)^{h\ZZ}$$ where at the second step, it is used that $Z$ is a finite spectrum.
	
	Thus the Postnikov filtration on the $\ZZ_p-\FF_p$-bimodule $\ZZ_p\otimes_{R} \FF_p$ gives a finite $\ZZ$-equivariant filtration with associated graded $\FF_p[\epsilon_{|v_1|+1}]$ for $p>2$ and $\FF_2[\epsilon_{|\eta|+1},\epsilon_{|v_1|+1}]$ for $p=2$. Taking the $\ZZ$-fixed points of this filtration, we get a finite filtration of the bimodule $\ZZ_p^{h\ZZ}\otimes_{j_{\zeta}} \FF_p^{h\ZZ}$ whose associated graded is $\FF_p^{h\ZZ}[\epsilon_{|v_1|+1}]$ for $p>2$ and $\FF_2^{h\ZZ}[\epsilon_{|\eta|+1},\epsilon_{|v_1|+1}]$ for $p=2$. As a $\ZZ_p^{h\ZZ}-\FF_p^{h\ZZ}$-bimodule, $\FF_p^{h\ZZ}$ corresponds to the functor $\Mod(\FF_p^{h\ZZ})^{\omega} \to \Mod(\ZZ_p^{h\ZZ})^{\omega}$ that is right adjoint to the base change from $\ZZ_p^{h\ZZ}$ to $\FF_p^{h\ZZ}$. Since $\epsilon_{|v_1|+1}$ is in odd degree, $\Uadd$ splits finite filtrations and sends suspension to $-1$, we learn that after applying $\Uadd$, the $f^*\circ \pi_*$ becomes null.
\end{proof}

\section{Topological cyclic homology}\label{sec:tc}

In this section we prove \Cref{thm:dgmpullback} and \Cref{thm:dgm-1connstrong}, which are refinements of \Cref{thm:dgmfixedpoints} of \Cref{thm:dgm-1conn}, and in particular extend the Dundas-Goodwillie-McCarthy theorem to certain $-1$-connective rings. This allows us to understand $K(j_{\zeta})$ in terms of the cyclotomic trace.

Given a ring $R$ giving $R$ a $\ZZ$ action is the same as giving an automorphism $\phi$ of $R$. $R^{h\ZZ}$ is then the pullback of the diagonal map $\Delta:R \to R\times R$ along the twisted diagonal $(1,\phi):R \to R\times R$. The idea for extending DGM to the nonconnective ring $R^{h\ZZ}$ is to use the work of Land--Tamme on the $K$ theory of pullbacks to relate the $K$-theory and $\TC$ of $R^{h\ZZ}$ to that of connective rings.

Recall that for any $\EE_1$-ring $R$, there is a standard $t$-structure on $\Mod(R)$, where a module is connective iff it is generated under colimits and extensions by $R$, and coconnective iff its underlying spectrum is.
\begin{lemma}\label{lem:tensorconnective}
	Let $R$ be a $-1$-connective $\EE_1$-ring. Then any $R$-module $M$ which is connective as a spectrum, $M$ is connective in the standard $t$-structure on $\Mod(R)$. In particular, for any right $R$-module $N$ whose underlying spectrum is connective, $M\otimes_RN$ is connective.
\end{lemma}
\begin{proof}
	Using the $t$-structure on $R$-modules, we obtain a cofibre sequence $\tau_{\geq 0}M \to M \to \tau_{<0}M$. $\tau_{\geq 0}M$ is $-1$-connective as an underlying spectrum since $R$ is, and it is built from $R$ via colimits and extensions. $\tau_{<0}M$ is coconnected as an underlying spectrum. Since $M$ is connective as a spectrum and $\tau_{\geq0}M$ is $-1$-connective as a spectrum, $\tau_{<0}$ is connective as well, so must be $0$. It follows that $M = \tau_{\geq0}M$ is connective in the $t$-structure.
	$M\otimes_RN$ is connective since it is built from $R\otimes_RN = N$ out of colimits and extensions and $N$ is connective.
\end{proof}
\begin{lemma}\label{lem:tensorbase}
	Suppose that $R \to R'$ is an $i$-connective map of $-1$-connective $\EE_1$-rings for $i\geq -1$, $M,N$ are right and left $R'$-modules that are connective in the standard t-structure. Then the map $M\otimes_R N \to M\otimes_{R'}N$ is $(i+1)$-connective.
\end{lemma}
\begin{proof}
	$M,N$ are built out of $R'$ under colimits and extensions, so it suffices to assume $M=N=R'$. Then we are trying to show that $R'\otimes_RR' \to R'$ is $i$-connective. This map has a section given by the left unit, so its fibre is the cofibre of the section. The cofibre of the unit map, $M'$, is $(i+1)$-connective by assumption. $M'\otimes_RR'$ is an extension of $M'\otimes_RR = M'$ by $M'\otimes_RM'$. $M'\otimes_RM'$ is $(2i+2)$-connective by \Cref{lem:tensorconnective}, so the result follows.
\end{proof}

The following proposition is due to Waldhausen \cite[Proposition 1.2]{waldhausen1984algebraic}\footnote{see also \cite[Lemma 2.4]{Land_2019}}, except he stated it for $\ZZ$-algebras, though the general proof is identical. We reproduce the proof below for convenience and future reference. The proposition is a precursor to trace methods.
\begin{proposition}[Waldhausen]\label{lem:walderrorterm}
	Let $f:R \to S$ be an $i$-connective map of connective $\EE_1$-algebras for $i\geq1$. Then $\fib(K(f))$ is $(i+1)$-connective, with $\pi_{i+i}\fib(K(f)) = \mathrm{HH}_0(\pi_0S;\pi_i\fib f)$.
\end{proposition}

\begin{proof}
	Since $i+1 \geq2$, it suffices by the Hurewicz theorem to show that $\fib(K(f))$ is $(i+1)$-connective and that $H_{i+1}\fib(K(f)) \cong \mathrm{HH}_0(\pi_0R;\pi_i\fib f)$. The nonpositive $K$-theory only depends on $\pi_0R$ for a connective ring \cite[Theorem 9.53]{BGT}, so by the Hurewicz theorem, it suffices to show the connectivity statement at the level of homology of $\BGL^+$. 
	
	Consider the map of homology Serre spectral sequences computing the homologies of $\BGL(R)$ and $\BGL(R)^{+}$ via the vertical maps in the diagram below:
	
	\begin{center}
		\begin{tikzcd}
			\BGL(R)	\ar[r]\ar[d] & \BGL(R)^{+}\ar[d]\\
			\BGL(S)	\ar[r] & \BGL(S)^+
		\end{tikzcd}
	\end{center}

	The signature of the $E_2$-term of the Serre spectral sequence for the left vertical map is $$H_p(\BGL(S);H_q(\fib \BGL(f))) \implies H_{p+q}(\BGL(R))$$ The first nonzero term in this spectral sequence is $H_0(\BGL (S);H_{i+1}(\fib \BGL (f)))$. Because $\GL$ and the infinite matrix ring $M$ only disagree on $\pi_0$ and $f$ is $i$-connective for $i\geq 1$, $\fib \GL f$ can be identified with $\fib Mf$. By the Hurewicz theorem, $H_{i+1}(\fib \BGL f)$ then agrees with $\pi_{i}\fib Mf = \pi_iM\fib f$, where we view $\fib f$ as a nonunital ring in order to make sense of $Mf$. Under this identification, the action of $\pi_0\GL(S)$ is identified with conjugation action of $\pi_0\GL(R)$ on $\pi_i M\fib f$. The trace then gives an isomorphism $\tr:H_0(\BGL (S);\pi_{i}(M\fib f)) \to \mathrm{HH}_0(\pi_0S;\pi_i\fib f)$.
	
	The $E_2$-term of the Serre spectral sequence for the right vertical map on the other hand is $H_*(\BGL(S)^+;H_*(\fib K(f)))$. Since $H_*(BGL(R)) \cong H_*(BGL(R))^{+}$, the map of Serre spectral sequences yields an isomorphism on abutments. $\BGL(R)^+ \to \BGL(S)^+$ comes from a map of spectra (the $1$-connective cover of $K$-theory), so the coefficient system for homology is trivial, and in the lowest degree $s$ in which $\fib(K(f))$ is nonzero, its homology is $H_s(\fib K(f))$, which must survive to the $E_{\infty}$-page for degree reasons. Since in the left vertical map Serre spectral sequence, there are no terms contributing to $H_s$ for $s\leq i$, we must then have that $H_s(\fib K(f)) = 0$ for $s\leq i$, giving the connectivity statement. By looking at the lowest nonvanishing terms in the spectral sequences and comparing the two spectral sequences, we then learn that $H_{i+1}(\fib K(f))) \cong H_0(\BGL (S);\pi_{i}(M\fib f)) \cong \mathrm{HH}_0(\pi_0S;\pi_i\fib f)$, proving the proposition.
\end{proof}

The following proposition is likely well known:

\begin{proposition}\label{prop:tclocalization}
	Let $S$ be a set of primes, and $R \to R'$ a map of connective $\EE_1$-rings that is an isomorphism on $\pi_0$. Then $\fib(\TC(R)\to \TC(R'))[S^{-1}] \cong \fib(\TC(R[S^{-1}]) \to R'[S^{-1}])$.
\end{proposition}
\begin{proof}
	The fibre of $\TC$ is in this case equivalent to the fibre of $K$-theory by \cite{dundas2012local}. Since $K$-theory is a filtered colimit preserving localizing invariant, it is a sheaf with respect to localizing at primes, with stalks the relative $\TC$ of the $p$-local rings. Thus it suffices to show the result for $R,R'$ $p$-local rings. In this case, $\TC$ is already $p$-local, so we can assume that $S = \{p\}$. The set of maps $R \to R'$ clear satisfy the $2$ out of $3$ property, so by considering the composite $R \to R' \to \pi_0R' = \pi_0R$, we can assume the map is $R \to \pi_0R$.
	
	Now since rationalization is $t$-exact, and because of the connectivity result in \Cref{lem:walderrorterm}, it suffices to show the result for the map $\tau_{\leq n}R \to \tau_{\leq n-1}R$. This is a square zero extension, so now the result follows from \cite[Theorem 5.15.1]{raskin2018dundas}.
\end{proof}

By taking the cospan $R \to R\times R \leftarrow R$ associated to $R^{h\ZZ}$ as discussed above, we see that the following result refines \Cref{thm:dgmfixedpoints}:
\begin{theorem}\label{thm:dgmpullback}
	Suppose we are given a map of cospans of connective $\EE_1$-rings that is levelwise $i$-connective for $i\geq 1$. Then for any truncating invariant $E$, the map on the pullbacks induces an $E$-equivalance, and an $i$-connective map on $\TC$.
\end{theorem}

\begin{proof}
	Let
	\begin{center}
		\begin{tikzcd}
			R_0\ar[r]\ar[d] &R_1 \ar[d] &\ar[l]R_2 \ar[d] \\
			S_0\ar[r]& S_1&\ar[l]S_2
		\end{tikzcd}
	\end{center}
be the map of cospans in consideration, and let $R_3,S_3$ denote the pullbacks.
	Applying \cite{Land_2019}, we obtain a pullback square:
	
	\begin{center}
		\begin{tikzcd}
			\pullback \Uloc(R_3) \ar[r]\ar[d] & \Uloc(R_0)\ar[d]\\
			\Uloc(R_2)\ar[r] & \Uloc(R_0\odot_{R_3}^{R_1}R_2)
		\end{tikzcd}
	\end{center}
	
	where the underlying spectrum of $R_0\odot_{R_3}^{R_1}R_2$ is $R_0\otimes_{R_3}R_2$, which is connective by \Cref{lem:tensorconnective}. Moreover one has a corresponding pullback square for the $S_i$. The maps $R_i \to S_i$ are $i$-connective for $i\geq 1$, so they induce an equivalence on $E$, and also the map on pullbacks is $(i-1)$-connective. The map $R_0\otimes_{R_3}R_2 \to S_0\otimes_{R_3}S_2$ is $i$-connective by \Cref{lem:tensorconnective}, and the map $S_0\otimes_{R_3}S_2\to S_0\otimes_{S_3}S_2$ is $i$-connective by applying both \Cref{lem:tensorconnective} and \Cref{lem:tensorbase} so the composite, which on underlying spectra agrees with $R_0\odot_{R_3}^{R_1}R_2 \to S_0\odot_{S_3}^{S_1}S_2$ is $i$-connective. It thus induces an equivalence on $E$ and a $(i+1)$-connective map on $\TC$.
	
	From the pullback square above, we then learn that $E(R_3) \to E(S_3)$ is also an equivalence, and that $\TC(R_3) \to \TC(S_3)$ is $n$-connective.
\end{proof}

We now apply \Cref{thm:dgmfixedpoints} to $j_{\zeta}$:

\begin{corollary}\label{cor:jdgm}
	Let $R = ko_2$ for $p=2$ and $\ell_p$ for $p>2$.
	There are pullback squares
	\begin{center}
		\begin{tikzcd}
			\pullback K(R^{h\ZZ}) \ar[r]\ar[d] & \TC(R^{h\ZZ})\ar[d]\\
			K(\ZZ_p^{h\ZZ})\ar[r] & \TC(\ZZ_p^{h\ZZ})
		\end{tikzcd}
	\end{center}
\begin{center}
	\begin{tikzcd}
			\pullback K(ko_2^{h\ZZ}) \ar[r]\ar[d] & \TC(ko_2^{h\ZZ})\ar[d] \\
K(\tau_{\leq 2}\SP_2^{h\ZZ})\ar[r] & \TC(\tau_{\leq2}\SP_2^{h\ZZ})
	\end{tikzcd}
\end{center}
where for $R$, the vertical maps are $(2p-2)$-connective for $p>2$ and $1$-connective for $p=2$, and for the second pullback square for $ko_2$, the vertical maps are $4$-connective. The vertical fibres are $p$-nil.
\end{corollary}
\begin{proof}
	The pullback squares follow immediately from \Cref{thm:dgmfixedpoints} and the fact that $\tau_{\leq3}ko_2 = \tau_{\leq 2}\SP_2$, $\tau_{\leq2p-3}\ell_2 = \ZZ_p$ which implies that the actions on those truncations are trivial.
	
	We now show that the vertical fibres are $p$-nil for the first square, as the proof for the second square is identical. Consider the square
	
	\begin{center}
		\begin{tikzcd}
			\TC(R^{h\ZZ})[\frac 1 p]\ar[r]\ar[d] & \TC(R^{h\ZZ}[\frac 1 p])\ar[d]\\
			\TC(\ZZ_p^{h\ZZ})[\frac 1 p]\ar[r] & \TC(\ZZ_p^{h\ZZ}[\frac 1 p])
		\end{tikzcd}
	\end{center}

	It will suffice to show this square is cartesian, since the map $R^{h\ZZ}[\frac 1 p] \to \ZZ_p[\frac 1 p]$ is an equivalence. By the proof of \Cref{thm:dgmpullback}, we have pullback squares
	
	\begin{center}
		\begin{tikzcd}
		 	{\pullback\TC(R^{h\ZZ})[\frac 1 p]} \ar[r]\ar[d] &\TC(R)[\frac 1 p]\ar[d]&	{\pullback\TC(R^{h\ZZ}[\frac 1 p]}) \ar[r]\ar[d] &\TC(R[\frac 1 p])\ar[d]\\
			\ar[r] \TC(R)[\frac 1 p] &\TC(R\odot_{\ell_p^{h\ZZ}}^{R\times\ell_p}R)[\frac 1 p] & {}	\ar[r] \TC(R[\frac 1 p] )&\TC(R[\frac 1 p]\odot_{R[\frac 1 p]^{h\ZZ}}^{R[\frac 1 p]\times R[\frac 1 p]}R[\frac 1 p])
		\end{tikzcd}
	\end{center}

	and similarly with $\ZZ_p$ replacing $R$. $R[\frac 1 p]\odot_{R[\frac 1 p]^{h\ZZ}}^{R[\frac 1 p]\times R[\frac 1 p]}R[\frac 1 p]$ agrees with $(R\odot_{R^{h\ZZ}}^{R\times R}R)[\frac1 p]$ because the underlying spectrum is a tensor product, and tensor products commute with inverting $p$. Thus we learn from these pullback squares that to show the pullback square we want is cartesian, it suffices to prove this with the pair $\ZZ_p^{h\ZZ},R^{h\ZZ}$ replaced by $\ZZ_p,R$ and $(\ZZ_p\odot_{\ZZ_p^{h\ZZ}}^{\ZZ_p\times\ZZ_p}\ZZ_p),(R\odot_{R^{h\ZZ}}^{R\times R}R)$. But these are pairs of connective rings with the same $\pi_0$, so the result follows now from \Cref{prop:tclocalization}.
\end{proof}

We now prove a variant of \Cref{thm:dgmpullback}, where the ring $R$ in question is $-1$-connective, but doesn't have to come from a pullback square. The idea is the same as before: to resolve $\Mod(R)$ by module categories of connective rings, only this time instead of the resolution coming to us from a pullback square, we construct one by hand. The result below is a refinement of \Cref{thm:dgm-1conn}.

\begin{theorem}\label{thm:dgm-1connstrong}
	Let $R \to S$ be an $1$-connective map of $-1$-connective rings such that $\pi_{-1}R$ is a finitely generated $\pi_0R$-module. Then for any truncating invariant $E$, $E(R) \to E(S)$ is an equivalence. Moreover, if $f$ is $n$-connective, then $\TC(R) \to \TC(S)$ is $(n-1)$-connective.
\end{theorem}
\begin{proof}
	Choose generators $x_1,\dots,x_n \in \pi_{-1}R$ as a $\pi_0R$-module. We will build an $R$-module $X$ whose cells correspond to the free monoid on the set $\{x_1,\dots,x_n\}$ such that its endomorphism ring is connective. We will then embed $\Mod(R)$ fully faithfully into $\Mod(\End(X))$, and show that the cofibre is also the module category of a connective ring. Doing the same for $S$, and comparing, we will obtain the result.
	
	To construct $X$, set $X_0=R$, and choose a free module of rank $n$ in degree $-1$ to hit the generators $x_i$ of $X$ in degree $-1$, and let $X_1$ be the cofibre. $\pi_{-1}X_1$ is canonically identified with $n$ copies of $\pi_{-1}X_0$ indexed on the $x_i$s. We can then repeat this process, constructing $X_i$ as the cofibre of a free module of rank $n^{i}$ hitting the generators of $X_{i-1}$, which are indexed on words of length $i$ in the $x_i$. Let $X = \colim_i X_i$. Note that $X$ is connective because the map $X_i \to X_{i+1}$ is zero on negative homotopy groups by construction. We claim that $\End(X) = \lim_i(\map(X_i,X))$ is connective. More generally, we will show that $\map(X,Y) = \lim_i\map(X_i,Y)$ is connective for any $Y$ that is connective as a spectrum. The individual terms $\map(X_i,Y)$ are connective because $Y$ is connective as a spectrum and $X_i$ are built out of finitely many cells of degree $0$. Thus it suffices to show that any map $X_i \to Y$ can be extended to $X_{i+1}$, so that the $\lim^1$-term that could potentially contribute to $\pi_{-1}(\map(X,Y))$ vanishes. But this follows since the obstructions to making an extension lives in $\pi_{-1}Y$, which vanishes.
	
	We now show that the thick subcategory generated by $X$ contains $R$. The $X_i$ filtration makes $X$ into a filtered $R$-module, with associated graded a free module on the free monoid generated by the $x_i$. We will construct a filtered self map $\sigma x_i:X \to X$ such that on the associated graded, $x_i$ is left multiplication by $x_i$\footnote{As pointed out to me by Robert Burklund, there is a universal example, the trivial square zero extension $\SP \oplus \bigoplus_1^n \Sigma^{-1}\SP$, which one can show by obstruction theory admits an $\EE_1$-map to $R$ for any $-1$-connective $R$ sending the classes in degree $-1$ to the classes $x_1,\dots,x_n$. This gives an alternate way to construct the module $X$ and self maps $\sigma x_i$ via basechange from the universal example.}. The obstruction to extending a filtered map defined until $X_{k-1}$ to $X_{k}$ at the $k$th step in the filtration is the map $\theta$ in the diagram below:
	
	\begin{center}
		\begin{tikzcd}
			& \Sigma^{-1}R^{n^k} \ar[d]\ar[dddr,"\theta"]&\\
			\cdots\ar[r] &X_{k-1} \ar[r]\ar[d] &X_k \ar[d,dashed] \ar[r] & \cdots\\
			\cdots\ar[r] &X_k \ar[r] &X_{k+1}\ar[d]\ar[r] &\cdots\\
			& & R^{n^{k+1}} &
		\end{tikzcd}
	\end{center}
	$\theta$ has to be null since the map $X_k \to X_{k+1}$ is $0$ on $\pi_{-1}$ by construction. Thus we can produce the dashed arrow in the diagram. Since $\pi_0$ of the space of nulhomotopies for an individual component $\Sigma^{-1}R$ is exactly $\pi_0R^{n^{k+1}}$, we can choose the nulhomotopy so that the map on the associated graded is as desired. 
	
	Now note that the cofibre of $X$ by all of the $\sigma x_i$s is just $R$, so that $R$ is in the thick subcategory generated by $X$. Let $\langle X\rangle$ be the thick subcategory of $\Mod(R)$ generated by $X$, and let $\langle X \rangle/\langle R \rangle$ denote the localization of $\langle X\rangle$ away from the thick subcategory generated by $R$. Morita theory gives an equivalence $\langle X \rangle = \Mod(\End(X))^{\omega}$, and we thus have a localization sequence
	 $$\Mod(R)^{\omega} \to \Mod(\End(X))^{\omega} \to\Mod(\End_{\langle X\rangle /\langle R\rangle}(X))^{\omega}$$ 
	 
	 We claim that $\End_{\langle X\rangle /\langle R\rangle}X$ is also connective. To understand this endomorphism ring, we observe that the $X_i$ are cofinal among perfect $R$-modules mapping to $X$. This means that $\End_{\langle X\rangle /\langle R\rangle}(X)$ is computed as $\colim_i \map(X,X/X_i)$, so it suffices to show $\map(X,X/X_i)$ is connective. But $X/X_i$ is connective as a spectrum, and as we have shown, $\map(X,Y)$ is connective whenever $Y$ is connective as a spectrum.
	
	Finally, we note that $X\otimes_RS$ has exactly the same properties as an $S$-module, and in fact $X\to X\otimes_RS$ is $i$-connective by \Cref{lem:tensorconnective}. However due to the possibility of $\lim^1$, this only guarantees that $\End(X) \to \End(X\otimes_RS)$ as well as the maps of endomorphisms in the cofibres are $(i-1)$-connective. Nevertheless, the contribution of the $\lim^1$-term to $\pi_0$ is square zero, so since truncating invariants are nil-invariant \cite[Theorem B]{Land_2019}, we learn that $E(\End(X)) \to E(\End(X\otimes_RS))$ is an equivalence, and similarly for the localized ring. By the localization sequence, the $E(R) \to E(S)$ is also an equivalence. Since $\End(X)\to \End(X\otimes_RS)$ is $(i-1)$-connective, it induces an $i$-connective map on $\TC$, so via the localization sequence, the original map $R \to S$ induces an $(i-1)$-connective map on $\TC$.
\end{proof}

\begin{remark}
	The connectivity bound for $\TC$ may not be optimal in the above theorem, and the finiteness hypothesis might not be necessary. Also, if one of the $x_i$ is chosen to be zero, $\End(X)$ vanishes on every additive invariant by the Eilenberg swindle. This means that we have really proven that the suspension of $\Uloc(R)$ is $\Uloc$ of a connective ring.
\end{remark}

\begin{question}
	To what extent do \Cref{thm:dgmpullback} and \Cref{thm:dgm-1connstrong} generalize?
\end{question}

For instance, can \Cref{thm:dgmpullback} be generalized to other finite limits of sufficiently connective ring maps? The results proven here are certainly not the most general: for example the methods of this section are capable of proving that for a sufficiently connective map of rings $R \to S$ with a $\ZZ^n$-action that is trivial in low degrees, $E(R^{h\ZZ^n}) \to E(S^{h\ZZ^n})$ is an equivalence for any truncating invariant $E$.

\section{The main theorems}

We now put together the results so far to prove the main theorems \ref{thm:K1localsphere} and \ref{thm:type2}, as well as \Cref{thm:compactk1local}, which are stated where they are proven for convenience.

	\begin{customthm}{A}
		$K(L_1^f \SP_p) \cong K(L_{K(1)}\SP)$, there is a cofibre sequence split on $\pi_*$
\begin{center}
	\begin{tikzcd}
		K(j_{\zeta}) \ar[r] & K(L_{K(1)}\SP) \ar[r] & \Sigma K(\FF_p)
	\end{tikzcd}
\end{center}
and a pullback square
\begin{center}
	\begin{tikzcd}
		\pullback K(j_{\zeta}) \ar[r] \ar[d] & \TC(j_{\zeta}) \ar[d]\\
		K(\ZZ_p) \ar[r]  & \TC(\ZZ_p^{h\ZZ})
	\end{tikzcd}
\end{center}
Let $F$ be the fibre of the map $\TC(j_{\zeta}) \to \TC(\ZZ_p^{h\ZZ})$. Then $F[\frac 1 p] = 0$. For $p>2$, $F$ is $(2p-2)$-connective and $\pi_{2p-2}(F/p) \cong \bigoplus_0^{\infty}\FF_p$. For $p=2$, $F$ is $1$-connective and $\pi_1F\cong \bigoplus_0^{\infty}\FF_2$.
\end{customthm}

\begin{proof}
	The first statement is just \Cref{prop:L1LK1}, and the cofibre sequence is \Cref{prop:K1localsequence}. We now show that the cofibre sequence is split on $\pi_*$. After inverting $p$, $K(j_{\zeta}) \to K(\ZZ_p^{h\ZZ})$ is an equivalence by \Cref{cor:jdgm}, so the map is null by \Cref{lem:nullona1invkthy}. Thus the cofibre sequence splits on $\pi_*$ after inverting $p$. Because $K_*(\FF_p) = K_*(\FF_p)[\frac 1 p]$ in positive degrees and is torsion, we obtain the desired result in degrees $\neq 1$. On $\pi_1$, one observes that $j_{\zeta}/(p,v_1)$ and $j_{\zeta}/(2,\eta,v_1)$ are zero in $K_0(j_{\zeta})$, so that the cofibre sequence is a short exact sequence on $\pi_1$, so it splits since $K_0(\FF_p)=\ZZ$ is projective.
	
	\Cref{cor:jdgm} gives a pullback square
	
	\begin{center}
		\begin{tikzcd}
			\pullback K(j_{\zeta}) \ar[r] \ar[d] & \TC(j_{\zeta}) \ar[d]\\
			K(\ZZ_p^{h\ZZ}) \ar[r]  & \TC(\ZZ_p^{h\ZZ})
		\end{tikzcd}
	\end{center}
	and gives the connectivity claims for $F$. By applying \Cref{thm:devissage} to $\ZZ_p^{h\ZZ}$, we learn that $K(\ZZ_p) \cong K(\ZZ_p^{h\ZZ_p})$, so that this pullback square agrees with the one in the theorem statement.
	
	It remains to prove the claims about the first nonvanishing homotopy group of $F$.
	
	To compute $\pi_1$ of the vertical fibre for $p=2$, we first observe from the second pullback square in \Cref{cor:jdgm} that it s the same as $\pi_1$ of the fibre of $K(\tau_{\leq 2}\SP^{h\ZZ}) \to K(\ZZ_2^{h\ZZ})$, where the action on $\tau_{\leq2}\SP$ is (necessarily) trivial. From Land--Tamme (see the proof of \Cref{thm:dgmpullback}), we get a pullback square 
	\begin{center}
		\begin{tikzcd}
		\pullback	\Uloc(\tau_{\leq2}\SP_2^{h\ZZ})\ar[r]\ar[d] &\Uloc(\tau_{\leq2}\SP_2)\ar[d]\\
		\Uloc(\tau_{\leq2}\SP_2)	\ar[r] & \Uloc(\tau_{\leq2}\SP_2\odot_{\tau_{\leq2}\SP_2^{h\ZZ}}^{\tau_{\leq2}\SP_2\times \tau_{\leq2}\SP_2}\tau_{\leq2}\SP_2)
		\end{tikzcd}
	\end{center}
	We claim that there is an equivalence of $\EE_1$-algebras $\tau_{\leq2}\SP_2\odot_{\tau_{\leq2}\SP_2^{h\ZZ}}^{\tau_{\leq2}\SP_2\times \tau_{\leq2}\SP_2}\tau_{\leq2}\SP_2 \cong \tau_{\leq2}\SP_2[z]$. To see this, we first apply the formula in \cite{LTformula2021} (see \cite[Example 4.9]{kcoconn}), which gives an equivalence of $\EE_1$-algebras $\tau_{\leq2}\SP_2\odot_{\tau_{\leq2}\SP_2[\epsilon_{-1}]}^{\tau_{\leq2}\SP_2[\epsilon_{0}]}\tau_{\leq2}\SP_2 \cong \tau_{\leq2}\SP_2[z]$. It suffices then to show that 
	
	$$\tau_{\leq2}\SP_2\odot_{\tau_{\leq2}\SP_2[\epsilon_{-1}]}^{\tau_{\leq2}\SP_2[\epsilon_{0}]}\tau_{\leq2}\SP_2 \cong \tau_{\leq2}\SP_2\odot_{\tau_{\leq2}\SP_2^{h\ZZ}}^{\tau_{\leq2}\SP_2\times \tau_{\leq2}\SP_2}\tau_{\leq2}\SP_2$$
	To do this, we use the observation in \cite{LTformula2021} (see also \cite[Section 4]{kcoconn}) that $R_1\odot_{R_2}^{R_4}R_3$ just depends on the span $R_1, R_3$ and the unital $R_1-R_3$-bimodule $R_4$. Thus it suffices to show that $\tau_{\leq 2}\SP[\epsilon_0]$ and $\tau_{\leq2}\SP\times \tau_{\leq2}\SP$ define isomorphic unital $\tau_{\leq2}\SP-\tau_{\leq2}\SP$-bimodules. But indeed, they are both symmetric bimodules, and $\tau_{\leq2}\SP[\epsilon_0]$ is a free $\EE_0$-$\tau_{\leq2}\SP$-algebra on $\epsilon_0$, so by sending $\epsilon_0$ to $(1,0) \in \pi_0 (\tau_{\leq2}\SP\times \tau_{\leq2}\SP)$ we get an isomorphism.
	
	The same argument for $\ZZ_2$ instead of $\tau_{\leq2}\SP_2$ shows that $\ZZ_2\odot_{\ZZ_2^{h\ZZ}}^{\ZZ_2\times \ZZ_2}\ZZ_2 = \ZZ_2[z]$. From \Cref{lem:walderrorterm}, we then learn that $K_2(\tau_{\leq2}\SP_2[z],\ZZ_2[z]) \cong \mathrm{HH}_0(\ZZ_2[z];\ZZ/2[z]) \cong \ZZ/2[z]$. By comparing the Land--Tamme pullback squares for $\tau_{\leq2}\SP^{h\ZZ}$ and $\ZZ_2^{h\ZZ}$, since $K_2(\tau_{\leq2}\SP_2,\ZZ_2) = \ZZ/2$, this shows that $K_1(\tau_{\leq2}\SP_2^{h\ZZ},\ZZ_2^{h\ZZ})$ is infinitely many copies of $\ZZ/2$.
	
	At odd primes, as before, it suffices to show that $\pi_{2p-1}K(\ell_p\odot_{\ell_p^{h\ZZ}}^{\ell_p\times\ell_p}\ell_p,\ZZ_p[z])/p$ is countably generated. To do this, we will first study it's underlying spectrum, which is a tensor product. We claim that the natural map $j_{\zeta}\otimes_{j_{\zeta}^{h\ZZ}}\ell^{h\ZZ} \to \ell_p$ is a $p$-completion, where we consider $j_{\zeta} \to \ell$ as a $\ZZ$-equivariant map with a trivial action on $j_{\zeta}$. To see this, we consider the commutative diagram below, where the horizontal arrows are given by $1$ minus the action of $1 \in \ZZ$.
	
	\begin{center}
		\begin{tikzcd}
			j_{\zeta}\ar[r]\ar[d] &j_{\zeta} \ar[r]\ar[d] & j_{\zeta}\ar[d]\ar[r] & \cdots \\
			\ell_p\ar[r] &\ell_p \ar[r] &\ell_p\ar[r] & \cdots
		\end{tikzcd}
	\end{center}
	The horizontal maps are ${j_{\zeta}}^{h\ZZ}$ and $\ell_p^{h\ZZ}$ module maps respectively. Since the action of $1 \in \ZZ$ on $\pi_*\ell_p$ is $1$ mod $p$, the horizontal maps are zero on $\pi_*$ mod $p$, so the colimit along the horizontal maps are zero $p$-adically. Moreover the fibres of each horizontal map is ${j_{\zeta}}^{h\ZZ}$ and $\ell_p^{h\ZZ}$ respectively. Thus we have produced a filtration of $j_{\zeta}$ as a ${j_{\zeta}}^{h\ZZ}$-module that basechanges $p$-adically to a filtration of $\ell_p$ as a $\ell_p^{h\ZZ}$ module, proving the claim.
	
	As a consequence, we obtain that the map $\ell_p\otimes_{{j_{\zeta}}^{h\ZZ}}j_{\zeta} \to \ell_p\otimes_{\ell_p^{h\ZZ}}\ell_p$ is an equivalence after $p$-completion. We now filter $j_{\zeta}$ via the homotopy fixed point filtration. Namely, $(\tau_{\geq*}\ell_p)^{h\ZZ_p}$ gives a filtration on $j_{\zeta}$. The map $j_{\zeta} \to \ell_p$ is also a filtered map, where $\ell_p$ is given the Postnikov filtration.

	In what follows, $C^*(\ZZ_p;R) = \colim_i(C^*(\ZZ/p^i;R))$ for $R$ an $\EE_{\infty}$-algebra denotes the algebra of continuous cochains on $\ZZ_p$ with coefficients in $R$. Up to $p$-completion, it is also given by the tensor product $R\otimes_{R^{h\ZZ}}R$, where $R$ has the trivial $\ZZ$-action.
	Taking the tensor product $j_{\zeta}/p\otimes_{{j_{\zeta}}^{h\ZZ}}j_{\zeta}$ in filtered rings gives a spectral sequence converging to the homotopy of the tensor product. We know that the tensor product is $(j_{\zeta}\otimes_{{j_{\zeta}}^{h\ZZ}}j_{\zeta})/p \cong C^*(\ZZ_p;j_{\zeta}/p)$. The associated graded of $j_{\zeta}$ is a $\ZZ$-algebra since it is the homotopy fixed points of the associated graded of $\ell_p$, which has the Postnikov filtration. We thus have isomorphisms $$\gr j_{\zeta}/p\otimes_{{\gr j_{\zeta}}^{h\ZZ}}\gr j_{\zeta}\cong \gr j_{\zeta}/p\otimes_{{\gr j_{\zeta}/p}^{h\ZZ}}\gr j_{\zeta}/p \cong C^*(\ZZ_p;\gr j_{\zeta}/p)$$ Thus we learn that the spectral sequence for $(j_{\zeta}\otimes_{{j_{\zeta}}^{h\ZZ}}j_{\zeta})/p$ degenerates. $\pi_* \gr j_{\zeta}/p = \FF_p[\zeta,v_1]$ since the Adams operations act trivially on $\pi_*\ell_p$ mod $p$, so the associated graded of the homotopy ring of this tensor product is $\FF_p[\zeta,v_1]\otimes C^*(\ZZ_p;\FF_p)$.
	
	The spectral sequence for $j_{\zeta}/p\otimes_{{j_{\zeta}}^{h\ZZ}}j_{\zeta}$ maps to the one coming from the tensor product of filtered rings $\ell_p/p\otimes_{{j_{\zeta}}^{h\ZZ}}j_{\zeta}$. The $E_1$-page of the spectral sequence for $\ell_p/p\otimes_{{j_{\zeta}}^{h\ZZ}}j_{\zeta}$ is the homotopy ring of $\gr \ell_p/p\otimes_{{\gr j_{\zeta}}^{h\ZZ}}\gr j_{\zeta}\cong \gr \ell_p/p\otimes_{{\gr \ell_p/p}^{h\ZZ}}\gr \ell_p/p \cong \gr \ell_p/p\otimes_{\FF_p} C^*(\ZZ_p;\FF_p)$. Thus the $E_1$-page is $\FF_p[v_1]\otimes_{\FF_p} C^*(\ZZ_p;\FF_p)$, so the map of spectral sequences is surjective and thus both spectral sequences degenerate. We also learn that $j_{\zeta}/p\otimes_{{j_{\zeta}}^{h\ZZ}}j_{\zeta} \to \ell_p/p\otimes_{j_{\zeta}^{h\ZZ}}j_{\zeta}$ is an isomorphism on $\pi_*$ in even degrees because at the level of $E_1$-pages it is the map $\FF_p[\zeta,v_1]\otimes_{\FF_p} C^*(\ZZ_p;\FF_p) \to \FF_p[v_1]\otimes_{\FF_p} C^*(\ZZ_p;\FF_p)$
	
	Using the formula for $\odot$ in the case the action is trivial as we did for $p=2$, we obtain an equivalence of $\EE_1$-rings $j_{\zeta}\odot_{j_{\zeta}^{h\ZZ}}^{j_{\zeta}\times j_{\zeta}}j_{\zeta} \cong j_{\zeta}[z]$. Thus we have an $\EE_1$-algebra map $$j_{\zeta}[z] \cong j_{\zeta}\odot_{j_{\zeta}^{h\ZZ}}^{j_{\zeta}\times j_{\zeta}}j_{\zeta} \to \ell_p\odot_{\ell_p^{h\ZZ}}^{\ell_p\times\ell_p}\ell_p$$ Since $\odot$ is the tensor product on underlying spectra, we learn that mod $p$, this map is an isomorphism in even degrees, so we learn that $\pi_*(\ell_p\odot_{\ell_p^{h\ZZ}}^{\ell_p\times\ell_p}\ell_p/p) = \FF_p[v_1,z]$ as a ring. It follows that $\mathrm{HH}_0(\ZZ_p[z];\pi_{2p-2} \ell_p\odot_{\ell_p^{h\ZZ}}^{\ell_p\times\ell_p}\ell_p/p) = \FF_p[z]$ so using \Cref{lem:walderrorterm}, we learn that $\pi_{2p-1}K(\ell_p\odot_{\ell_p^{h\ZZ}}^{\ell_p\times\ell_p}\ell_p,\ZZ_p[z])/p$ is $\FF_p[z]$, which indeed is countably generated.
	
	Using the formula for $\odot$ in the case the action is trivial as we did for $p=2$, we obtain an equivalence of $\EE_1$-rings $j_{\zeta}\odot_{j_{\zeta}^{h\ZZ}}^{j_{\zeta}\times j_{\zeta}}j_{\zeta} \cong j_{\zeta}[z]$. Thus we have an $\EE_1$-algebra map $$j_{\zeta}[z] \cong j_{\zeta}\odot_{j_{\zeta}^{h\ZZ}}^{j_{\zeta}\times j_{\zeta}}j_{\zeta} \to \ell_p\odot_{\ell_p^{h\ZZ}}^{\ell_p\times\ell_p}\ell_p$$ Since $\odot$ is the tensor product on underlying spectra, we learn that mod $p$, this map is an isomorphism in even degrees, so we learn that $\pi_*(\ell_p\odot_{\ell_p^{h\ZZ}}^{\ell_p\times\ell_p}\ell_p/p) = \FF_p[v_1,z]$ as a ring. It follows that $\mathrm{HH}_0(\ZZ_p[z];\pi_{2p-2} \ell_p\odot_{\ell_p^{h\ZZ}}^{\ell_p\times\ell_p}\ell_p/p) = \FF_p[z]$ so using \Cref{lem:walderrorterm}, we learn that $\pi_{2p-1}K(\ell_p\odot_{\ell_p^{h\ZZ}}^{\ell_p\times\ell_p}\ell_p,\ZZ_p[z])/p$ is $\FF_p[z]$, which indeed is countably generated.
\end{proof}

\begin{question}\label{qst:null}
	Is the boundary map $K(\FF_p) \to K(j_{\zeta})$ in \Cref{thm:K1localsphere} null?
\end{question}

We saw in the above proof that the map is null after inverting $p$, so that \Cref{qst:null} is essentially a $p$-adic question.

Next, we give a formula for $K(\Sp^{\omega}_{T(1)})$:

\begin{theorem}\label{thm:compactk1local}
	There is a cofibre sequence split on $\pi_*$
	\begin{center}
		\begin{tikzcd}
			K(j_{\zeta}\otimes \Sp_{\geq1}^{\omega}) \ar[r] & K(\Sp_{T(1)}^{\omega}) \ar[r] & \Sigma K(\FF_p)
		\end{tikzcd}
	\end{center}
	and a pullback square
	\begin{center}
		\begin{tikzcd}
			\pullback K(j_{\zeta}\otimes \Sp_{\geq1}^{\omega}) \ar[r] \ar[d] & \TC(j_{\zeta}) \ar[d]\\
			K(\FF_p) \ar[r]  & \TC(\ZZ_p^{h\ZZ})
		\end{tikzcd}
	\end{center}
\end{theorem}

\begin{proof}
	First note that
	
	\begin{center}
		\begin{tikzcd}
			K(j_{\zeta}\otimes \QQ)\ar[r] &K(L_{K(1)}\SP\otimes \QQ) \ar[r] & 0
		\end{tikzcd}
	\end{center}
	is a cofibre sequence since $j_{\zeta}\otimes \QQ \cong L_{K(1)}\SP\otimes \QQ$. Thus combining this with cofibre sequence from \Cref{thm:K1localsphere} via the localization sequences for rationalization, we get the claimed cofibre sequence. 
	
	To obtain the pullback square, we again consider what happens when we rationalize. Then $j_{\zeta}\otimes \QQ \cong \QQ_p^{h\ZZ}$, so 
	
	\begin{center}
		\begin{tikzcd}
			\pullback K(j_{\zeta}\otimes \QQ) \ar[r]\ar[d] & \ar[d]\TC(j_{\zeta}\otimes \QQ)\\
			K(\QQ_p)\ar[r] &\TC(\QQ_p^{h\ZZ}) 
		\end{tikzcd}
	\end{center}
	is not only a pullback square, but the vertical fibres vanish. By combining this pullback square with the one for $j_{\zeta}$ in \Cref{thm:K1localsphere}, we obtain the claimed pullback square. 
	
	Finally, the splitting of the cofibre sequence at the level of $\pi_*$ follows exactly as in the proof of \Cref{thm:K1localsphere}. Namely, after inverting $p$, the cofibre sequence becomes a split cofibre sequence by \Cref{lem:nullona1invkthy}, which shows that the cofibre sequence is split on $\pi_*$ in degrees $\neq1$ since $K(\FF_p)$ is $p'$-torsion in degrees $\neq 0$. In degree $1$, one gets a short exact sequence on homotopy groups since $K_0(\FF_p) \to K_0(j_{\zeta}\otimes \Sp_{\geq1}^{\omega})$ is null, and this short exact sequence splits since $K_0(\FF_p) \cong \ZZ$ is projective.
\end{proof}

\begin{customthm}{D}
	There is a fibre sequence $X \to K(\Sp_{\geq 2}) \to K(\FF_p)$ split on $\pi_*$, where $X$ is the total fibre of the square
	\begin{center}
		\begin{tikzcd}
			\TC(\SP_p)\ar[r]\ar[d] &\TC(\ZZ_p) \ar[d]\\
			\ar[r] \TC(j_{\zeta})& \TC(\ZZ_p^{h\ZZ})
		\end{tikzcd}
	\end{center}
	\begin{itemize}
		\item For $p>2$, $X$ is $(2p-3)$-connective, so $K_0(\Sp_{\geq 2}) = \ZZ$ with generator $[\SP/(p,v_1)]$.
		\item For $p=2$, $X$ is connective with $\pi_0X = \bigoplus_0^{\infty}\ZZ/2$, and the torsion free quotient of $K_0(\Sp_{\geq 2})$ is generated by $[\SP/(2,\eta,v_1)]$.
	\end{itemize}
\end{customthm}

\begin{proof}
	Consider the diagram of cofibre sequences given by tensoring the first localization sequence of rings with $j_{\zeta}$ and applying $K$-theory. We use \Cref{prop:K1localsequence} to identify the lower sequence.
	
\begin{center}
	\begin{tikzcd}
		\pullback K(\Sp_{\geq 2})\ar[r]\ar[d] & \ar[r]\ar[d] K(\SP_p)& \ar[d]K(L_1^f\SP_p) \\
		K(\FF_p)	\ar[r] &K(j_{\zeta}) \ar[r] &K(L_{K(1)}\SP)
	\end{tikzcd}
\end{center}
The right vertical map is an equivalence by \Cref{thm:K1localsphere}, so the left square is a pullback square. Thus the fibre $K(\Sp_{\geq2}) \to K(\FF_p)$ is the fibre $K(\SP) \to K(j_{\zeta})$ which by comparing the DGM squares for $\SP$ with the pullback square in \Cref{thm:K1localsphere} yields the pullback square. The claims about $X$ come from the claims about the vertical fibres in the pullback square of \Cref{thm:K1localsphere}. Namely, $\fib(\TC(\SP_p) \to \TC(\ZZ_p))$ is $(2p-2)$-connective by \Cref{lem:walderrorterm}, so combining this with the connectivity bound in \Cref{thm:K1localsphere} gives the claim.

The splitting on $\pi_*$ of the cofibre sequence follows from a similar argument as in the proof of \Cref{thm:K1localsphere}. Namely, the map $K(\SP_{2})[\frac 1 p] \to K(\FF_p)[\frac 1 p]$ is an equivalence, because $\fib(\TC(\SP_p) \to \TC(\ZZ_p))[\frac 1 p] \cong \fib(\TC(\SP_p[\frac 1 p]) \to \TC(\ZZ_p[\frac 1 p])) =0$ by \Cref{prop:tclocalization}, and $\fib(\TC(j_{\zeta}) \to \TC(Z_p^{h\ZZ}))[\frac 1 p] = 0$ by \Cref{cor:jdgm}. In particular, the cofibre sequence is split after inverting $p$, and since $\pi_iK(\FF_p) = \pi_iK(\FF_p)[\frac 1 p]$ in nonzero degrees, the cofibre sequence is split on homotopy in nonzero degrees. In degree $0$, it is a short exact sequence homotopy groups since the map $K_0(\FF_p) \to K_0(j_{\zeta})$ is null, and must be split since $K_0(\FF_p)$ is projective.

The remaining thing to justify is that $\SP/(p,v_1)$ for $p>2$ and $\SP/(2,\eta,v_1)$ for $p=2$ are generators of the torsion free quotient of $K_0$. The map $K(\Sp_{\geq2}) \to K(\FF_p)$ is on $K_0$ exactly this torsion free quotient, and the generator of $K_0$ was seen to be as claimed in the proof of \Cref{prop:K1localsequence}.
\end{proof}
\section{The K-theory sheaf}\label{sec:interpretations}

In this section we define the $K$-theory sheaf $K^{\Delta}$, and explain how it classifies stable tensor ideals, refining the classification of thick tensor ideals of the Balmer spectrum. After doing so, we extract some consequences of our main theorems.

Recall that if $C$ is a small rigid symmetric monoidal stable category\footnote{it is sufficient for $C$ to be monoidal, in which case $K^{\Delta}(C)$ is a sheaf of $\EE_1$-rings}, the Balmer spectrum of $C$, $\Spec^{\Delta}(C)$, allows one to classify thick tensor ideals of $C$ \cite{balmer2005spectrum}. To each open set $O$ of the Balmer spectrum, there is a finite localization $L_O(C)$ given by localizing away from objects whose support doesn't intersect $O$. By associating to each $O$ the algebraic $K$-theory of $L_O(C)$, we obtain the \textit{$K$-theory sheaf} $K^{\Delta}(C)$ on $\Spec^{\Delta}(C)$, which is a sheaf of $\EE_{\infty}$-rings.

Knowing $K^{\Delta}(C)$ allows one to in particular refine the classification of thick tensor ideals to a classification of stable tensor ideals. This is due to the following elementary result of Thomason:

\begin{proposition}[Thomason \cite{thomason1997classification}]\label{prop:stabletensorideals}
	Given a small stable category $C$, there is a bijection between dense stable subcategories $C' \subset C$ and subgroups of $K_0(C)$, given by the assignment $C' \mapsto K_0(C')\subset K_0(C)$.
\end{proposition}

Here, a dense subcategory of a stable category is one such that the inclusion induces an equivalence on idempotent completions. It follows from \Cref{prop:stabletensorideals} that every stable tensor ideal is given by a thick subcategory $C' \subset C$ and a submodule of $K_0(C')$ as a $K_0(C)$-module. $K_0(C')$ can be extracted from $K^{\Delta}(C)$ as follows: since $K$-theory commutes with filtered colimits, we can assume that $C'$ is a compact stable tensor ideal. Its support is then a closed set $Z_{C'}$ of the Balmer spectrum, and we let $O_{C'}$ denote the open complement. We have a cofibre sequence $K(C') \to K(C) \to K(L_{O_{C'}}(C))$, allowing us to extract $K_0(C')$.

We essentially ran the above process of extracting $K(C')$ in the proof of \Cref{thm:type2} to obtain a description of $K(\Sp_{\geq2})$ as the fibre of $K(\SP_p) \to K(L_1^f\SP_p)$. The following is a consequence of \Cref{thm:type2}.
\begin{corollary}\label{cor:densesubcattype2}
	The dense stable subcategories of $\Sp_{\geq2}^{\omega}$ for $p>2$ are in bijection with subgroups of $\ZZ$, and the dense stable subcategories of $\Sp_{\geq2}^{\omega}$ at the prime $2$ are in bijection with subgroups of $\ZZ\oplus \bigoplus_0^{\infty}\FF_2$.
\end{corollary}

The following corollary is a consequence of the fact that the map $K_0(\Sp_{\geq1}) \to K_0(\Sp_{K(1)}^{\omega})$ is surjective (in fact it is an isomorphism).

\begin{corollary}\label{cor:k1locallift}
	Any compact $K(1)$-local spectrum is the $K(1)$-localization of a type $1$ spectrum.
\end{corollary}

\begin{proof}
	One observes that the $K(1)$-local spectra which are $K(1)$-localizations of type $1$ spectra are a stable subcategory of $\Sp_{K(1)}^{\omega}$ corresponding to the subgroup of $K_0$ that is the image of $K_0(\Sp_{\geq1})$.
\end{proof}

Note that apriori all that is clear is that a compact $K(1)$-local spectrum is a retract of the $K(1)$-localization of a finite type $1$ spectrum.

\begin{question}
	Given a compact $K(1)$-local spectrum, is there a way of finding a lift of it to $\Sp_{\geq1}$?
\end{question}

Next we interpret \Cref{thm:pinverted} in the corollary below.
We can say that a stable subcategory $C' \subset C$ is \textit{$p$-saturated} if $\oplus_1^p X \in C' \implies X \in C'$. 

\begin{corollary}[Burklund--Levy]\label{cor:psatsubcat}
	The nonzero $p$-saturated stable subcategories of $\Sp_{p}$ are specified by a type $n$ and a $\ZZ[\frac 1 p]$-submodule of $\ZZ[\frac 1 p]$.
\end{corollary}

\section{Constructing type 2 spectra}\label{sec:const}

In \Cref{thm:type2}, it was shown that $K_0(\Sp_{\geq2}) \simeq \ZZ\oplus \bigoplus_0^{\infty}\FF_2$ at the prime $2$. Here we explain how to construct type $2$ spectra representing the $2$-torsion classes. The key point is understanding the boundary maps in the $K$-theory of a localization sequence, which is the purpose of the \Cref{lem:boundarymap} below.

\begin{construction}\label{const:boundarycomputation}
	Let $C \xrightarrow{i} D \xrightarrow{\pi} E$ be a localization sequence, and let $f:d \to d$ for $d \in D$ be a map. Suppose $f$ has the property that its cofibre is in $C$. Then if $\Map(d,d)_f$ denotes the connected component containing $f$, taking the cofibre gives a map $\cof:\Map(d,d)_f \to \mathrm{B}\Aut(\cof f)$.
	
	On the other hand, since $\cof f$ vanishes in $E$, composing with $\pi$ gives a map $\pi:\Map(d,d)_f \to \Aut(\pi d)$.
\end{construction}

Recall also that given an object $c \in C$, there is a canonical map $\mathrm B\Aut(c) \to C^{\simeq} \to \Omega^{\infty}K(C)$

\begin{lemma}\label{lem:boundarymap}
In the situation of \Cref{const:boundarycomputation}, the diagram below commutes up to a sign, where $\delta$ is the boundary map associated to the localization sequence.

\begin{center}
	\begin{tikzcd}	\Map(d,d)_f\ar[r,"\pi"]\ar[dr,"\cof"]& \Aut(\pi d)\ar[r] & \Omega^{\infty+1}K(E) \ar[d,"\delta"]\\
	&	\mathrm{B}\Aut(\cof f)\ar[r] & \Omega^{\infty}K(C)
	\end{tikzcd}
\end{center}
\end{lemma}

\begin{proof}
	Recall (eg see \cite[Definition 3.4]{hebestreit2022localisation}) that $\Omega^{\infty}K(C)$ can be modeled via the $Q$-construction $\Omega |\Span(C)|$, where $\Span(C)$ is the category of spans of objects in $C$, with composition given by pulling back. The natural map $C^{\simeq} \to \Omega^{\infty}K(C)$ sends $x$ to the span $0 \leftarrow x\rightarrow 0$.
	
	Consider the diagram below in $|\Span(C)|$ where an arrow with a perpendicular line indicates the direction of a span as a map in $\Span(C)$. All of the $2$-cells are the obvious ones, except for the one labeled $\pi f^{-1}$, where $\pi f^{-1}$ is the difference between the $2$-cell and the obvious one.
	\begin{center}
		\begin{tikzcd}
			&& 0 \\
			& 0 & {\Sigma \pi d} & 0 \\
			0 & 0 & {\Sigma \pi d} & 0 & 0 \\
			& {\pi d} & 0 & {\pi d} \\
			& {} & 0 & {} \\
			&& 0
			\arrow["\shortmid"{marking}, from=2-3, to=3-3]
			\arrow[from=4-2, to=5-3]
			\arrow["\shortmid"{marking}, from=4-2, to=3-1]
			\arrow[from=4-4, to=5-3]
			\arrow["\shortmid"{marking}, from=4-4, to=3-5]
			\arrow["\shortmid"{marking}, from=2-4, to=3-5]
			\arrow[from=2-4, to=1-3]
			\arrow[from=2-2, to=1-3]
			\arrow["\shortmid"{marking}, from=2-2, to=3-1]
			\arrow["\shortmid"{marking}, from=3-2, to=3-1]
			\arrow[from=3-2, to=3-3]
			\arrow[from=3-4, to=3-3]
			\arrow["\shortmid"{marking}, from=3-4, to=3-5]
			\arrow[from=2-3, to=1-3]
			\arrow["\shortmid"{marking}, from=4-3, to=3-3]
			\arrow[from=4-3, to=5-3]
			\arrow[bend left = 50, from=6-3, to=3-1]
			\arrow["\shortmid"{marking}, bend right = 50, from=6-3, to=3-5]
			\arrow["{\pi f^{-1}}"', shift right=5, shorten <=30pt, shorten >=30pt, Rightarrow, from=5-2, to=5-4]
		\end{tikzcd}\end{center}
	
	Because the boundary is sent to $0$, this diagram represents a map $S^2 \to |\Span(C)|$. Since it is given by conjugating $\pi f^{-1}$ by the triangles in the diagram above, it represents the image of $f$ in $\Omega^{\infty+1} K(E)$. To compute $\delta$ of this, consider the lift of the diagram below to $|\Span(D)|$:
	
		\begin{center}
		\begin{tikzcd}
			&& 0 \\
			& \cof f & {\Sigma d} & \cof f \\
			\cof f & \cof f & {\Sigma d} &\cof f & 0 \\
			& {d} & 0 & {d} \\
			& {} & 0 & {} \\
			&& 0
			\arrow["\shortmid"{marking}, from=2-3, to=3-3]
			\arrow[from=4-2, to=5-3]
			\arrow["\shortmid"{marking}, from=4-2, to=3-1]
			\arrow[from=4-4, to=5-3]
			\arrow["\shortmid"{marking}, from=4-4, to=3-5]
			\arrow["\shortmid"{marking}, from=2-4, to=3-5]
			\arrow[from=2-4, to=1-3]
			\arrow[from=2-2, to=1-3]
			\arrow["\shortmid"{marking}, from=2-2, to=3-1]
			\arrow["\shortmid"{marking}, from=3-2, to=3-1]
			\arrow[from=3-2, to=3-3]
			\arrow[from=3-4, to=3-3]
			\arrow["\shortmid"{marking}, from=3-4, to=3-5]
			\arrow[from=2-3, to=1-3]
			\arrow["\shortmid"{marking}, from=4-3, to=3-3]
			\arrow[from=4-3, to=5-3]
			\arrow[bend left = 50, from=6-3, to=3-1]
			\arrow["\shortmid"{marking}, bend right = 50, from=6-3, to=3-5]
			\arrow["{ f^{-1}}"', shift right=5, shorten <=30pt, shorten >=30pt, Rightarrow, from=5-2, to=5-4]
	\end{tikzcd}\end{center}
Here $f^{-1}$ is the map of spans identifying $d$ with the fibre of $d \to \cof f$. Note there is a subtlety about making the diagram lift the previous one: the maps $\Sigma d \to \Sigma \pi d$ are given by multiplication by $f$, and the maps $d \to \pi d$ are the canonical ones.
This lift is a diagram in the shape of $D^2$ and has the property that its boundary $S^1$ lives in $|\Span(C)|$. Thus the boundary $S^1$ represents $\delta$ applied to the map $S^2 \to |\Span(E)|$. But by composing the composable maps in the boundary, we find that the boundary is canonically homotopic to (up to a sign) $\{\cof f\} \to C^{\simeq} \to \Omega^{\infty}K(C)$. By the naturality of this construction in $f$, we have shown the diagram commutes as claimed.
\end{proof}

Given the lemma, we now find explicit generators for $K_1(\SP^{h\ZZ})$. There is an equivalence $\SP^{h\ZZ} \simeq \SP[\epsilon_{-1}]$, and a localization sequence

\begin{center}
	\begin{tikzcd}
		\Mod(\SP[\epsilon_{-1}])\ar[r] &\Mod(\SP[x]) \ar[r] & \Mod(\SP[x^{\pm1}])
	\end{tikzcd}
\end{center}
where we identify $\SP[\epsilon_{-1}] = \End_{\SP[x]}(\SP[x]/x)$. From \Cref{lem:walderrorterm}, and examining the long exact sequence on homotopy groups, we learn that $K_1(\SP[\epsilon_{-1}]) = K_1(\SP) \oplus \coker(\ZZ/2[x] \to \ZZ/2[x^{\pm1}])$ where $\ZZ/2[x]$ and $\ZZ/2[x^{\pm1}]$ are $\mathrm{HH}_0(\pi_0R; \pi_1 R)$, where $R = \SP[x],\SP[x^{\pm}]$. In the proof of \Cref{lem:walderrorterm}, the Hochschild homology term is coming from $H_0(\BGL(\pi_0R);\pi_2(\BGL(R))$, which is isomorphic to $H_0(\BGL_1(\pi_0R);\pi_2(\BGL_1(R))$ since the bimodule is symmetric. Thus those $K_2$ classes come from the classes $\eta x^{i} \in \pi_2 \mathrm{B}\Aut(R)$. Thus to find representatives of the $K_1$ classes, we need to compute the boundary map $K_2(\SP[x^{\pm1}]) \to K_1(\SP[\epsilon_{-1}])$ on these classes, but this exactly what \Cref{lem:boundarymap} is made to do. Namely, note that $\eta x^{-i} \in \pi_1 \Aut(\SP[x^{\pm 1}])$ lifts to $\pi_1$ of the component of $\Map_{\SP[x]}(\SP[x],\SP[x])$ containing the map $x^i$, by composing with the automorphism $\eta$ on the target.

Taking the cofibre, we get a nontrivial element, which we call $g_i$ of $\pi_1(\mathrm{B}\Aut(\cof f)) = \pi_0(\Aut(\cof f))$. $g_i$ is the map obtained as the horizontal cofibre of the diagram

\begin{center}
\begin{tikzcd}
	{\SP[x]} & {\SP[x]} \\
	{\SP[x]} & {\SP[x]}
	\arrow["\eta", shorten <=10pt, shorten >=10pt, Rightarrow, from=2-1, to=1-2]
	\arrow[Rightarrow, no head, from=1-2, to=2-2]
	\arrow["{x^i}", from=1-1, to=1-2]
	\arrow[Rightarrow, no head, from=1-1, to=2-1]
	\arrow["{x^i}"', from=2-1, to=2-2]
\end{tikzcd}
\end{center}

Since $g_i$ is nontrivial, for $i=1$ the only possibility is that $g_1 = 1 + \beta_x \eta$. In general, we can use the fact that $\eta{x^{-i}}$ is in the image of the analogous localization sequence for $\SP[x^i]$ to learn that $g_i = 1 + \beta_{x^i}\eta$.

We can describe these maps in terms of $\SP^{h\ZZ}$. Let $X_i$ be the module corresponding to $\SP[x]/x^i$. It has a cellular filtration by the other $X_j$, defined inductively by observing that $X_i$ can be constructed as the cofibre of the map $\zeta_{i-1}:\Sigma^{-1}\SP^{h\ZZ} \to X_{i-1}$ given by hitting the generator in $\pi_{-1}$.

$\zeta_{i}$ extends to a self map of $\zeta_i:\Sigma^{-1}X_i \to X_i$ corresponding to $\beta_{x^i}$\footnote{Explicitly, this map is the composite $X_i \to \Sigma X_{2i} \to \Sigma X_{i}$, where we view $X_i$ as the first and last $i$ cells of $X_{2i}$.}, which we give the same name. Thus we have proven:

\begin{lemma}
	$K_1(\SP^{h\ZZ},\SP)$ is generated by $[g_n]$, where $g_n$ is the automorphisms $g_n=1+\eta \zeta_i:X_i \to X_i$. 
\end{lemma}

The map $K_1(\SP^{h\ZZ},\SP) \to K_1({j_{\zeta}}^{h\ZZ},\ZZ_2^{h\ZZ})$ is an isomorphism (see \Cref{cor:jdgm}), so the latter is generated by automorphisms of the same name. By \Cref{thm:type2}, the image of these class along the map $K_1({j_{\zeta}}^{h\ZZ}) \to K_1(L_{K(1)}\SP_2) \simeq K_1(L_1^f\SP_2) \to K_0(\Sp^{\omega}_{\geq2})$ are the generators of $2$-torsion classes in $K_0$. 

To actually compute the representatives in $K_0$, we first observe that these $2$-torsion classes are in the kernel of the map $K_1(L_{K(1)}\SP) \to K_1(L_{K(1)}\SP\otimes \QQ)$, because $\eta$ is. It follows from the localization sequence that these classes lift to $K_1(\Sp_{T(1)}^{\omega})$. To actually produce lifts, we observe that the diagram below (thought of in $\Mod(L_{K(1)}\SP^{\omega})$) commutes up to homotopy:

\begin{center}
	\begin{tikzcd}
		X_i\ar[r,"1"]\ar[d,"2"] & X_i\ar[d,"2"]\\
	X_i	\ar[r,"g_i"] & X_i
	\end{tikzcd}
\end{center}

Let $\bar{g_i}$ denote an automorphism of $X_i/2$ obtained by taking vertical cofibres. Since $1$ represents the trivial element of $K_1$ and the class is $2$-torsion, $\bar{g}_i$ represents the same $K_1$-class as $\bar{g_i}$ by additivity of $K$-theory, but constitutes a lift to $K_1(\Sp_{T(1)}^{\omega})$. Note that $\bar{g}_i$ is \textit{not} $g_i\otimes \cof 2$, as the latter has trivial $K_1$ class, as $\cof2$ is $0$ in $K_0(L_{K(1)}\SP)$.

The desired $K_0$ classes are then obtained as the image via the map $K_1(\Sp_{T(1)}^{\omega}) \to K_0(\Sp_{\geq2})$ from the localization sequence $\Sp_{\geq2} \to \Sp_{\geq 1} \to \Sp_{T(1)}^{\omega}$. This can be again computed by \Cref{lem:boundarymap}, but we need to use a trick to account for the fact that the self map we need to cofibre by is no longer in degree $0$. This trick is an instantiation of the rotation invariance phenomenon studied in \cite{rotation}.

\begin{lemma}\label{lem:rotationinvariance}
	Let $C$ be a stable category. Then the map $\Uloc(C) \to \Uloc(C[x_2^{\pm1}])$ is naturally a split inclusion.
\end{lemma}
\begin{proof}
	There is a localization sequence $C[x_2]^{x_2-\text{nil}} \to C[x_2] \to C[x_2^{\pm 1}]$. Because $C$ is canonically a retract of $C[x_2]$, it will suffice to show that the map $F:C[x_2]^{x_2-\text{nil}} \to C[x_2] \to C$ is naturally null on $\Uloc$. There is a cofibre sequence of functors in $\Fun(C[x_2]^{x_2-\text{nil}})$
	
	\begin{center}
		\begin{tikzcd}
			\Sigma^2U\ar[r,"x_2"] &U \ar[r] & F
		\end{tikzcd}
	\end{center}
	where $U$ is the underlying functor, and $x_2$ is the natural transformation given by multiplication by $x_2$. Since $\Uloc$ is additive and $\Sigma^2U$ and $U$ induce the same map after applying $\Uloc$, it follows that $F$ is null upon applying $\Uloc$.

\end{proof}

Out next goal is to lift $X_i\otimes \cof 2$ to $\Sp_{\geq1}$. Recall that $X_i$ was constructed as the cofibre of a map $\zeta_{i-1}:\Sigma^{-1}L_{K(1)}\SP \to X_{i-1}$. Let us fix a $v_1$-self map on $\cof2$, so that a power of $v_1$ will indicate a power of that particular self map. After tensoring with $\cof 2$, we can lift maps from $\Sp_{T(1)}^{\omega}$ to $\Sp_{\geq 1}$ after sufficient composition with $v_1$. Thus we can inductively construct finite type $1$ spectra $\tilde{X}_i$ such that its $T(1)$-localization is $X_i\otimes \cof 2$ and it has a lift of $v_1^{?}\zeta_i$ so that the cofibre is $\tilde{X}_{i+1}$.

Choose a $v_1$ self map on $\tilde{X}_{i+1}$, and note that $\bar{g}_iv_1^{j_i}$ lifts to a self map $\tilde{g}_i$ of $\tilde{X}_{i+1}$ for $j_i$ sufficiently large. 

\begin{proposition}
	The boundary map $K_1(\Sp_{T(1)}^{\omega}) \to K_0(\Sp_{\geq2})$ sends $\bar{g}_i$ to $[\cof \tilde{g}_i] - [\cof v_1^{j_i}]$, where the cofibres are taken as self maps of $\tilde{X_i}$. Thus these are a basis of the $2$-torsion of $K_0(\Sp_{\geq2})$, and these along with $\cof(2,\eta,v_1)$ generate $K_0(\Sp_{\geq 2})$.
\end{proposition}
\begin{proof}
	The later statements follow from the claim about the boundary map by applying \Cref{thm:type2}, so we will just prove the claim about the boundary map. By \Cref{lem:rotationinvariance}, it suffices to do so after tensoring the localization sequence with $\SP[x_2^{\pm1}]$. Let $u_1$ denote the $v_1$-self map of $X_i\otimes \cof 2$, except shifted into degree $0$. It is an automorphism so we have $[\bar{g}_i] = [\bar{g}_iu_1^{j_i}]-[u_1^{j_i}]$. But $\bar{g}_iu_1^{j_i}$ and $u_1^{j_i}$ lift to self maps of $\tilde{X}$, so by applying \Cref{lem:boundarymap} and observing that $x_2$ is an automorphism so can be ignored when taking cofibres, we learn that the boundary is $[\cof \tilde{g}_i ] - [\cof v_1^{j_i}]$.
	\end{proof}

We now run this consruction explicitly for $g_1$. Here $\cof 2$ has a $v_1^4$-self map, $g_1$ is $1+\eta \zeta$, and $\sigma$ is a lift of $\zeta v_1^4$. since $\eta \sigma$ is $2$-torsion, so we can form the map $\overline{\eta \sigma}: \Sigma^8\SP/2 \to \SP \to \SP/2$, where the first map in the composite is given by a nulhomotopy of $2\eta \sigma$. $\tilde{g}_1$ is then given by the automorphism of $\cof 2$ named $v_1^4+\overline{\eta \sigma}$, so the first $2$-torsion $K_0$ class is $[\SP/(2,v_1^4+\overline{\eta\sigma})] - [\SP/(2,v_1^4)]$.

Even though we can explicitly construct representing $K_0$ classes, the computation of $K_0(\Sp_{\geq2})$ remains somewhat inexplicit. For example, given a type $2$ spectrum $X$, where $p>2$, it does not obviously give a way of building $X$ out of $\SP/(p,v_1)$ via cofibre sequences. On the other hand, any type $1$ spectrum has an explicit way of being built out of $\SP/p$: namely its cell decomposition naturally decomposes into Moore spectra. Thus we ask:

\begin{question}\label{qst:explicit}
	Is there an explicit way, given a type $2$ spectrum, to build it out of representatives of the generating $K_0$ classes via cofibre sequences?
\end{question}

\begin{question}\label{qst:p2explicit}
	At the prime $2$, given a type $2$ spectrum, is there a way of computing its $K_0$ class?
\end{question}

One possible approach to \Cref{qst:p2explicit} would be to try to understand the isomorphisms $K_0(\Sp_{\geq 2}) \cong \im(K_1(\Sp_{K(1)}^{\omega}))\cong  K_1(\Sp_{K(1)}^{\omega})/K_1(\Sp_{\geq1})$.

\section{Euler characteristics}\label{sec:eulerchar}
We turn to studying the torsion free part of $K_0(\Sp_{\geq n})$, and use \Cref{thm:compactk1local} an answer to \cite[Problem 16.4]{hovey1999morava} at height $1$. We study natural homomorphisms out of this torsion free part called Euler characteristics. We explain how the image of the Euler characteristic $\chi_{\BP\langle n\rangle}$ is an obstruction to small type $n+1$ spectra such as Smith--Toda complexes existing. Using the existence of spectra such as $\ko$ and $\tmf$, we compute the image of $\chi_{\BP\langle n \rangle}$ at heights $\leq 2$, and conjecture what the answer is in general.

Recall from \cite{mahowaldrezk} that a spectrum $X$ is said to be fp of type at most $n$ if it is bounded below, $p$-complete, and $X\otimes Y$ is $\pi$-finite for any $Y \in \Sp_{\geq n+1}$. Let $\fp_{n}$ denote the full subcategory of $\Sp$ consisting of fp spectra of type at most $n$. Note that $\fp_{-1}$ is the category of $p$-torsion $\pi$-finite spectra.

\begin{lemma}\label{lem:fp-1}
	$K(\fp_{ -1}) \simeq K(\FF_p)$ and $K(\fp_{0}) \simeq K(\ZZ_p)$.
\end{lemma}
\begin{proof}
	The $t$-structure on spectra is bounded on $\fp_{0}$ and $\fp_{-1}$ with hearts finitely generated discrete $\ZZ_p$-modules and $p$-nil discrete $\ZZ_p$-modules respectively. The result then follows from the nonconnective theorem of the heart \cite{antieau2018ktheoretic} (which is a refinement of \cite{Barwick_2015}), and Quillen's devissage \cite{quillenhigherktheory}.
\end{proof}

\begin{remark}
	We could alternatively have used \Cref{thm:devissage} to prove \Cref{lem:fp-1}, but the method above seems more direct.
\end{remark}

Tensoring sets up a pairing $\chi:K(\fp_n)\otimes K(\Sp_{\geq n+1}) \to K(\fp_{-1}) \simeq K(\FF_p)$, which we call the \textit{Euler characteristic}. If we fix a class $[X] \in K_0(\fp_n)$, we obtain a map $\chi_{X}:K(\Sp_{\geq n+1}) \to K(\FF_p)$. On $\pi_0$, identifying $K_0(\FF_p) \simeq \ZZ$, this map takes $Y$ to $\Sigma_i (-1)^i\log_p |\pi_i(X\otimes Y)|$.

Euler characteristics are in general nontrivial homomorphisms. For example, if $\BP\langle n \rangle$ is the ($p$-completed) truncated Brown--Peterson spectrum, then it is easy to see that on the generalized Moore spectrum $\SP/(p^{i_0},v_1^{i_1},\dots v_n^{i_n})$, $\chi_{\BP\langle n\rangle}$ takes the value $\Pi i_j$. In fact $\chi_{\BP\langle n\rangle}$ is shown in \cite{pinverted} to give an equivalence $K(\Sp_{\geq n+1})[\frac 1 p] \to K(\FF_p)[\frac 1 p]$. In particular, the image of the map on $\pi_0$ integrally is generated by some power of $p$. 

Euler characteristics can be used to obstruct the existence of small type $n$ spectra. For example, because $\ko_2/\eta \cong \ku_2$ and $\ko_2$ and $\ku_2$ are in $\fp_1$, the image of $\chi_{\ku_2}$ is a multiple of $2$ (in fact it is exactly $2$). It follows that the Smith--Toda complex $\SP/(2,v_1)$ cannot exist, because $\chi_{\ku_2}(\SP/(2,v_1))$ would be $1$. This example shows the image of $\chi_{\BP\langle n \rangle}$ is an obstruction to Smith--Toda complexes and other small type $n+1$ spectra from existing. Below we compute the image of $\chi_{\BP\langle n\rangle}$ at low heights.

\begin{proposition}\label{prop:chibpn}
	The table below lists the image of $\chi_{\BP\langle n\rangle}:K_0(\Sp_{\geq n+1}) \to \ZZ$ at some low heights. In all these cases, there exists $X \in \fp_n$ such that $\chi_X$ has image $\ZZ$.
	\begin{center}
\begin{tabular}{ |c|c|c|c|c|} 
	\hline
	&\multicolumn{4}{|c|}{prime}\\
	\hline
	 $n$& $2$ & $3$ & $5$ &$>5$ \\
	\hline
	$0,-1$ & $\ZZ$& $\ZZ$& $\ZZ$&$\ZZ$\\
	$1$& $2\ZZ$ & $\ZZ$ &$\ZZ$ &$\ZZ$\\ 
	$2$& $8\ZZ$ & $3\ZZ$ &$\ZZ$& $\ZZ$ \\ 
	$3$& ?& ?& ?& $\ZZ$\\ 
	\hline
\end{tabular}
\end{center}
\end{proposition}

\begin{proof}
	In all the cases in the table where the answer is $\ZZ$, there exists a Smith--Toda complex $V(n)$ \cite{toda1971spectra}, which has the property that $V(n)\otimes BP\langle n\rangle = \FF_p$, so $1$ is indeed in the image. 
	
	For $n=1,p=2$, we observe that $ko_2$ is a type $1$ fp spectrum such that $ko_2/\eta = BP\langle 1\rangle$ and $ko_2\otimes \SP/(2,\eta,v_1) = \FF_2$. It follows that $2[ko_2] = [BP\langle 1 \rangle ]$ in $K_0(\fp_1)$ and that the image of $\chi_{ko_2}$ is $1$. Thus the image of $\chi_{\BP\langle 1\rangle} = 2\chi_{ko_2}$ is $2$.
	
	For $n=2,p=2$, we use \cite{lawson2012commutativity}, which shows that $2$-adically there is an $\EE_{\infty}$-map $\tmf \to \tmf_1(3) = BP\langle 2\rangle$ realizing at the level of mod $2$ cohomology the quotient $\cA\sslash E(2) \to \cA\sslash \cA(2)$. It follows that $\tmf_1(3)\otimes_{\tmf}\FF_2$ is the dual of $\cA(2)\sslash E(2)$, which is an exterior algebra on three generators in even degrees. Since $\tmf$ is connective with $\pi_0(\tmf) = \ZZ_2$, the map $ K_0(\tmf) \to K_0(\FF_2) = \ZZ$ is an isomorphism, and we learn that $\tmf_1(3)$ is perfect over $\tmf$ and $[\tmf_1(3)]$ is the class $8 \in K_0(\tmf)$. $\tmf$ is fp of type $2$: in \cite{bhattacharya2016class}, a type $2$ spectrum $Z$ was constructed with a $v_2^1$-self map with the property that $Z/v_2\otimes \tmf \simeq \FF_2$. It follows that $\chi_{\tmf}$ has image $\ZZ$, and that $\chi_{\BP\langle 2\rangle} = 8\chi_{\tmf}$.
	
	For $n=2,p=3$, \cite{behrens2004existence} constructed a $v_2^1$ self map on the spectrum $Y(2)\otimes V(1)$, where $Y(2)$ is the $8$-skeleton of $\SP\sslash \alpha_1$. Moreover, there is an equivalence $Y(2)\otimes \tmf \simeq \BP\langle 2 \rangle \oplus \Sigma^{8}\BP\langle 2\rangle$ (see Remark 2.3 of ibid.). It follows that $(Y(2)\otimes V(1))/v_2$ has $\chi_{\tmf}$ equal to $2$, and $3[\tmf] = 2[\BP\langle 2\rangle]$. Thus the image of $\chi_{\BP\langle 2\rangle}$ must be $3\ZZ$, and the image of $\chi_{\BP\langle2\rangle \oplus \Sigma \tmf}$ is $\ZZ$.
\end{proof}

\begin{question}\label{qst:imagechibpn}
	What is the generator of the image of $\chi_{\BP\langle n\rangle}:K_0(\Sp_{\geq n+1}) \to \ZZ$? Is it the size of the maximal finite $p$-subgroup of the height $n$ Morava stabilizer group?
\end{question}

Indeed, in the cases studied in \Cref{prop:chibpn}, the result agrees with the size of the maximal $p$-subgroup of the Morava stabilizer group at that height (see \cite{hewett1995finite}).

\begin{question}\label{qst:dualityfptype}
	Does there always exist an $X$ fp of type $n$ such that $\chi_X:K_0(\Sp_{\geq n+1}) \to \ZZ$ is surjective?
\end{question}

The truth of \Cref{qst:dualityfptype} would suggest that fp spectra and finite spectra are dual in the sense that the extent of the failure of small type $n+1$-spectra to exist corresponds exactly to that of the failure of $\BP\langle n\rangle$ to be a regular fp type $n$ ring closest to the sphere.

Now we turn to answer a question of Hovey--Strickland \cite{hovey1999morava}. They considered two homomorphisms, $\chi, \xi$ out of $K_0(\Sp_{K(n)}^{\omega})$. At height $n$, prime $p$, suppose that $M_*$ is a graded $p$-torsion graded abelian group of periodicity $|v_n|p^i$. Then $\len M_*$ is defined to be $\frac 1 {p^i}\sum_0^{|v_n|p^i-1}\log_p|M_*|$. $\chi:K_0(\Sp_{K(n)}) \to \ZZ$ is defined as $[X] \mapsto \len \pi_{\text{even}}(E(n)_*\otimes X)-\len\pi_{\text{odd}}(E(n)_*\otimes X)$ and $\xi: K_0(\Sp_{K(n)}) \to \ZZ[\frac 1 p]$ is defined as $[X] \mapsto \len \pi_{\text{even}}(X)-\len\pi_{\text{odd}}(X)$. They then ask:

\begin{problem}[{\cite[Problem 16.4]{hovey1999morava}}]\label{qst:hoveystrickland}
	What is the relationship between $\chi$ and $\xi$? Is $\chi$ an isomorphism? If not, is $\QQ\otimes \chi$ an isomorphism? Can one say anything about the higher $K$-theory of $\Sp_{K(n)}^{\omega}$?
\end{problem}

In \Cref{thm:compactk1local}, we described $K(\Sp_{K(1)}^{\omega})$ as a spectrum. We now answer the rest of \Cref{qst:hoveystrickland} for $n=1$.

\begin{proposition}\label{prop:hovstrick}
	For $n=1$, $\xi=0$, and $\chi$ is an isomorphism.
\end{proposition}

\begin{proof}
	The generator of $K_0(\Sp_{K(1)}^{\omega})\simeq \ZZ$ is $L_{K(1)}\SP/p$, so we need merely check the result on the generator. $\pi_*L_{K(1)}/p$ is an exterior algebra on $\pi_*K(1)$ on $\zeta$, which is an odd degree class, so we see that $\xi$ evaluates to $0$. $E(1)/p$ is $K(1)$, so $\chi(\SP/p)=1$.
\end{proof}

In \cite{pinverted}, we answer most of \Cref{qst:hoveystrickland} at all heights.

\begin{proposition}[\cite{pinverted}]\label{prop:pinvhoveystrick}
	${\chi}$ is an isomorphism after inverting $p$ and $\xi$ is $0$.
\end{proposition}

\begin{remark}\label{rmk:chis}
	In fact $\chi$ is closely related to $\chi_{\BP\langle n-1\rangle}$: there is a commutative diagram
	
	\begin{center}
		\begin{tikzcd}
			K_0(\Sp_{\geq n})\ar[r,"\chi_{\BP\langle n-1\rangle}"]\ar[d] & \ZZ \ar[d,equal]\\
			K_0(\Sp_{K(n)})\ar[r,"\chi"] & \ZZ
		\end{tikzcd}
	\end{center}
	Indeed, it suffices to check this rationally, where it can be easily checked on a generalized Moore spectrum by the results of \cite{pinverted}.
\end{remark}




\begin{question}\label{qst:typenknlocal}
	Is the map $K_0(\Sp_{\geq n}) \to K_0(\Sp_{K(n)}^{\omega})$ an isomorphism after quotienting by $p$-torsion?\footnote{One can also ask: is this true with $T(n)$ replacing $K(n)$?}
\end{question}

Because of \Cref{rmk:chis}, a positive answer to \Cref{qst:typenknlocal} would imply that the remaining part of \Cref{qst:hoveystrickland}, determining the image of $\chi$, is equivalent to \Cref{qst:imagechibpn}.

\bibliographystyle{alpha}
\bibliography{ref}
\end{document}

%% file: template.tex
\usepackage{geometry}
\usepackage{graphicx}
\usepackage{todonotes}
\usepackage{amsmath}
\usepackage{amsfonts}
\usepackage{amssymb}
\usepackage{tikz-cd}
\usepackage{comment}
\usepackage{stmaryrd}
\usepackage{fullpage}
\usepackage{epstopdf}
\usepackage{amsthm}
\usepackage{mathrsfs} 
\usepackage{tikz}
\usepackage{dsfont}
\usetikzlibrary{matrix}
\usepackage{mathtools}
\usepackage[hidelinks]{hyperref}
\usepackage{cleveref}
\usepackage{xcolor}
\usepackage{enumitem}
\usepackage{tensor}
\usepackage[utf8]{inputenc}
\usepackage{csquotes}
\usepackage[english]{babel}
\usepackage{mymacros}

\hypersetup{
	colorlinks,
	linkcolor={red!50!black},
	citecolor={blue!50!black},
	urlcolor={blue!80!black}
}

\tikzset{
	rot90/.style={anchor=south, rotate=90, inner sep=.5mm}
}
\tikzset{
	rot45/.style={anchor=south, rotate=-45, inner sep=.5mm}
}

\newtheorem{theorem}{Theorem}[section]

\newtheorem{lemma}[theorem]{Lemma}
\newtheorem{definition}[theorem]{Definition}
\newtheorem{proposition}[theorem]{Proposition}
\newtheorem{corollary}[theorem]{Corollary}

\newtheorem{innercustomthm}{Theorem}
\newenvironment{customthm}[1]
{\renewcommand\theinnercustomthm{#1}\innercustomthm}
{\endinnercustomthm}

\theoremstyle{definition}
\newtheorem{construction}[theorem]{Construction}
\newtheorem{question}[theorem]{Question}
\newtheorem{problem}[theorem]{Problem}
\newtheorem{remark}[theorem]{Remark}